\documentclass[12pt]{article}
\usepackage[english]{babel}
\usepackage{amsthm,amsfonts, amsbsy, amssymb,amsmath,graphicx}
\usepackage{graphics}

\usepackage[all]{xy}
\usepackage{array}
\usepackage{blkarray}

\newtheorem{theorem}{Theorem}
\newtheorem{lemma}{Lemma}
\newtheorem{corollary}{Corollary}

\newtheorem{proposition}{Proposition}

\theoremstyle{definition}
\newtheorem{definition}{Definition}
\newtheorem{example}{Example}

\theoremstyle{remark}
\newtheorem{remark}{Remark}

\newcommand{\Z}{\mathbb Z}
\newcommand{\V}{{\mathcal V}}

\newcommand{\bega}{\left(\begin{array}}
\newcommand{\ena}{\end{array}\right)}

\newenvironment{block}{}{}

\date{}

\title{Parity on based matrices}
\author{Igor Nikonov}

\begin{document}

\maketitle

\begin{abstract}
A parity is a labeling of the crossings of knot diagrams which is compatible with
Reidemeister moves. We define the notion of parity for based matrices --- algebraic objects introduced by V.~Turaev~\cite{T} in his research of virtual strings. We present the reduced stable parity on based matrix which gives a new example of a parity of virtual knots.
\end{abstract}

\section*{Introduction}

V.O. Manturov~\cite{M1} defined a parity as a rule to assign labels $0$ and $1$ (considered as elements of $\Z_2$) to the crossings of knot diagrams in a way compatible with Reidemeister moves. Capability to distinguish odd and even crossings allows to treat the crossings differently when transforming knot diagrams or calculating their invariants.
The notion of parity has proved to be an effective tools in knot theory. It allows to strengthen knot invariants, to prove minimality theorem and to construct (counter)examples~\cite{IMN}.

Then a natural question is to describe existing parities on virtual knot. So far the only known parity on virtual knots was the Gaussian parity~\cite{M1}. On the other hand, virtual knots can be viewed as knots in surfaces modulo isotopy, Reidemeister moves and (de)stabilization. The parities for knots in a fixed surface are well known~\cite{IMN}. They comes from homology of the underlying surface: the parity of a knot crossing is the homology class of a knot half at the crossing. This is the point where a link to the based matrices is established.

Based matrices introduced by V.~Turaev~\cite{T} in his research of virtual strings (also known as flat knots). One can assign to a diagram of a flat knot the matrix whose elements are intersection numbers of diagram halves at the crossing of the diagram. Reidemeister moves induce transformations of based matrices. The reduced form of based matrices is a powerful invariant of flat knots, and some other known invariants can be deduced from it.

Later analogues of based matrices were introduced for virtual knots~\cite{T2}, singular flat knots~\cite{C,H}, framed and long virtual knots~\cite{P} and virtual links~\cite{F}.

The aim of the present paper is to define parity on based matrices and find new parities for flat and virtual knots.

The paper is structured as follows. Section~\ref{sect:definition} recalls definitions of flat knots, parity and based matrices. In Section~\ref{sect:parity_based_matrices} we introduce the definition of parity and parity functor on based matrices and present an example of such a parity --- stable parity. The stable parity determines a new parity on flat and virtual knots.

\section{Definitions}\label{sect:definition}

\subsection{Virtual and flat knots}

A {\em  $4$-graph} is any union of four-valent graphs and {\em trivial components}, i.e. circles  considered as graphs without vertices and with one (closed) edge. A {\em virtual diagram} is an embedding of a $4$-graph into plane so that each vertex of the graph is marked as either {\em classical} of {\em virtual} vertex. At a classical vertex one a pair of opposite edges (called {\em overcrossing}) is chosen. The other pair of opposite edges at the vertex is called {\em undercrossing}. The vertices of the diagram are called also {\em crossings}.

Virtual crossings of a virtual diagram  are usually drawn circled. The undercrossing of a classical vertex is drawn with a broken line whereas the overcrossing is drawn with a solid line (see Fig.~\ref{fig:virtual_trefoil}). A diagram without virtual crossings is {\em classical}.

\begin{figure}[h]
\centering
\includegraphics[width=0.15\textwidth]{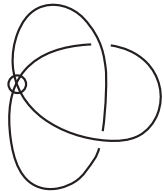}
\caption{Virtual trefoil diagram with two classical and one virtual crossings}\label{fig:virtual_trefoil}
\end{figure}

\begin {definition}
A {\em virtual link}~\cite{K} is an equivalence class of virtual diagrams modulo Reidemeister moves and detour moves (Fig.~\ref{fig:reidemeister_moves}).

\begin{figure}[h]
\centering
\includegraphics[width=0.3\textwidth]{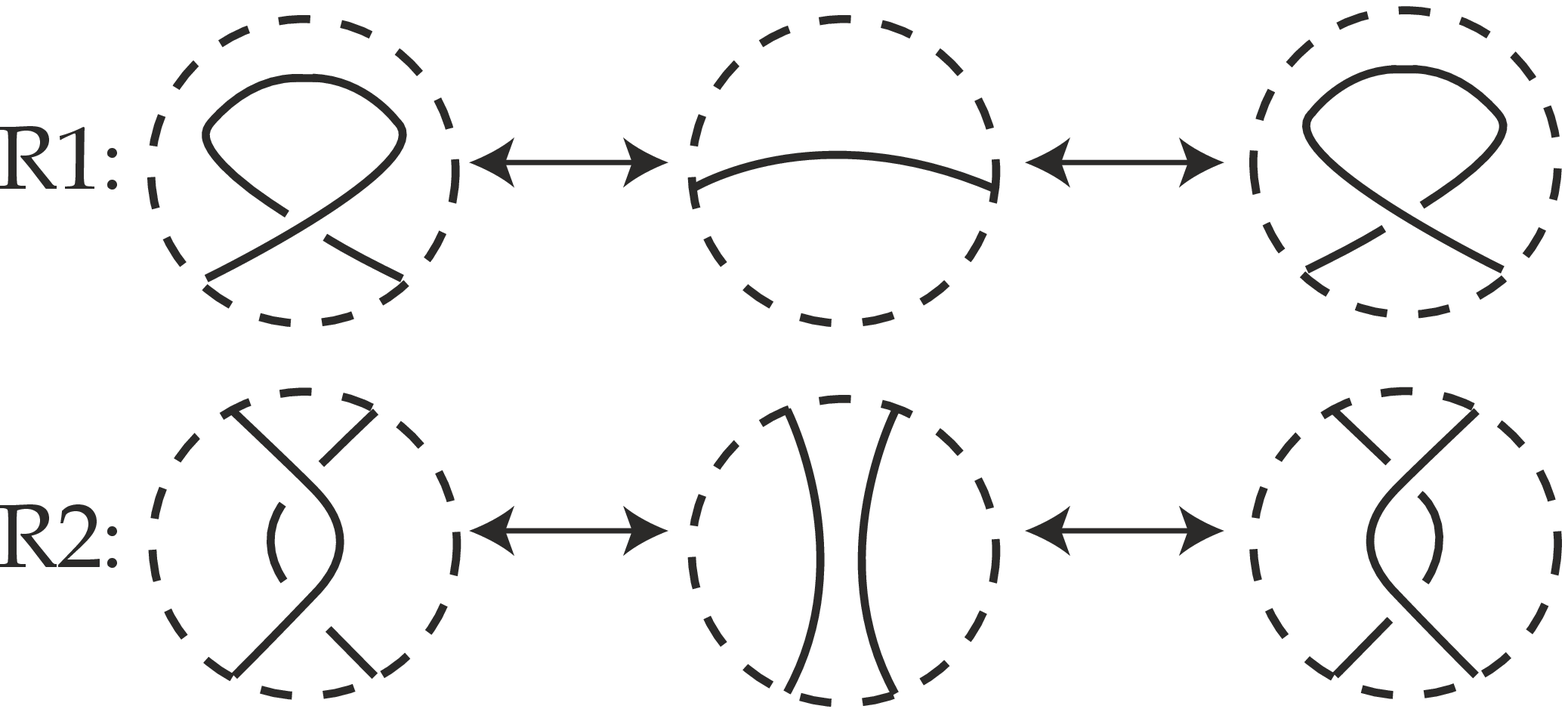}\qquad
\includegraphics[width=0.3\textwidth]{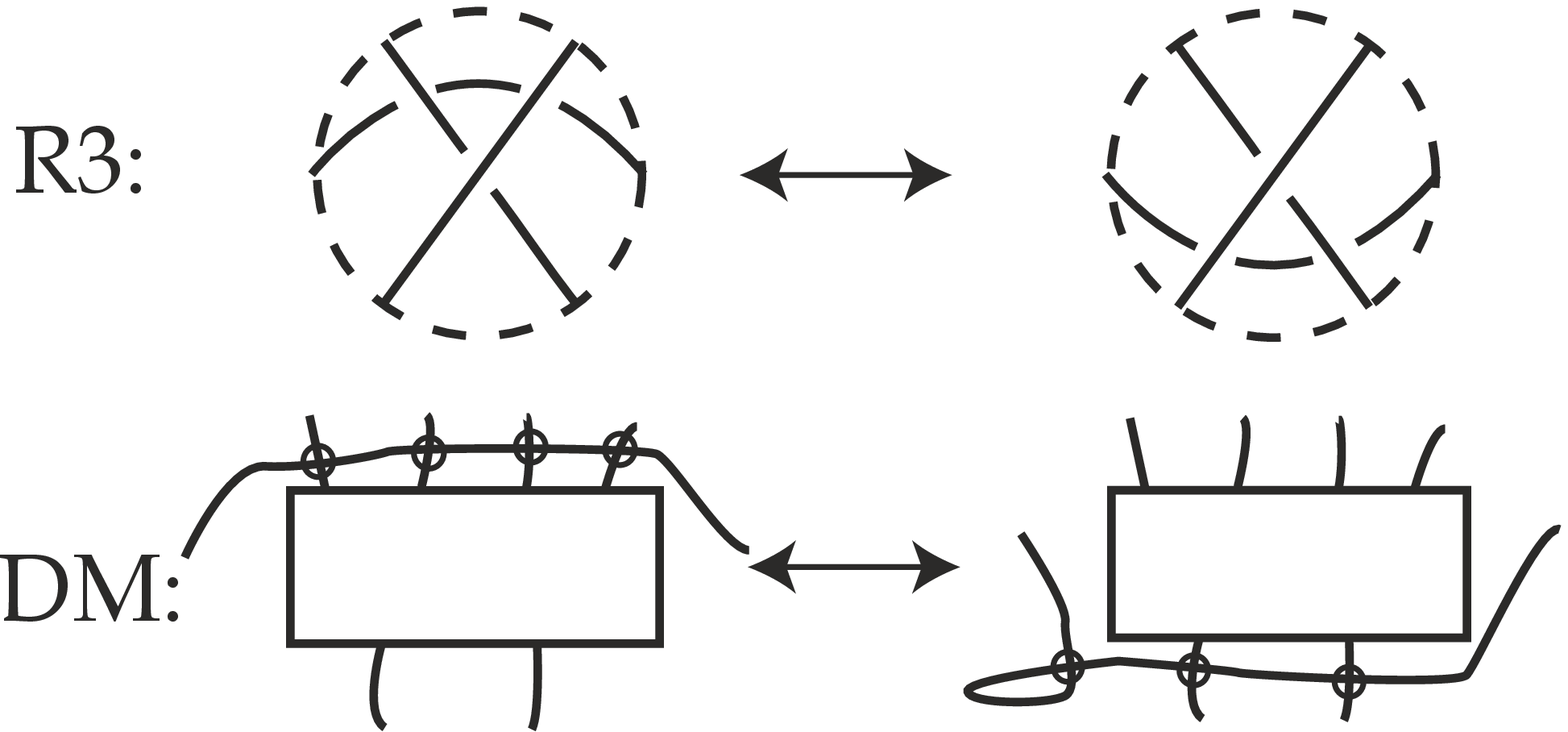}
\caption{Reidemester moves and the detour move}\label{fig:reidemeister_moves}
\end{figure}
\end {definition}

A {\em unicursal component} is a minimal set of diagram edges which is closed under passing from an edge to its opposite (at some end of the edge) edge. A diagram with one unicursal component is called a {\em diagram of a virtual knot}.

Below we assume that the unicursal components of diagrams are oriented.

There is another description of virtual knots.

A {\em virtual link} can be viewed as an equivalence class of pairs $(S,D)$ where $S$ is a closed oriented surface and $D$ is a diagram in $S$ whose crossings are all classical~\cite{KK}. The equivalence relation is generated by diagram isotopies, classical Reidemeister moves and stabilizations (see Fig.~\ref{fig:stabilization}) which change the surface.

\begin{figure}[h]
\centering
  \includegraphics[width=0.5\textwidth]{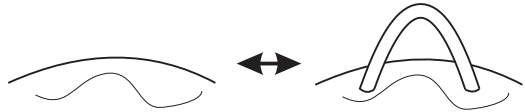}\\
  \caption{Stabilization move}\label{fig:stabilization}
\end{figure}

For example, the virtual trefoil can be given by a diagram in the torus (Fig.~\ref{fig:virtual_trefoil_surface}).

\begin{figure}[h]
\centering
\includegraphics[width=0.25\textwidth]{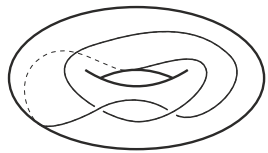}
\caption{Virtual trefoil}\label{fig:virtual_trefoil_surface}
\end{figure}

If one excludes the stabilization moves she gets a knot theory in a given surface.

If one admits crossing switch transformations (Fig.~\ref{fig:crossing_switch}) of virtual diagrams, i.e. neglects the over-undercrossing structure, one gets the theory of {\em flat knots}.

\begin{figure}[h]
\centering
\includegraphics[width=0.25\textwidth]{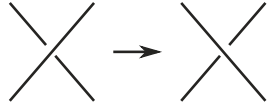}
\caption{Crossing switching}\label{fig:crossing_switch}
\end{figure}

Equivalently, flat knots (links) are equivalence classes of {\em flat diagrams} ($4$-valent graphs whose vertices are virtual crossing and classical crossings without under-overcrossing structure, Fig.~\ref{fig:flat_knot_diagram}) modulo flat Reidemeister moves (Fig.~\ref{fig:flat_reidemeister_moves}).

\begin{figure}[h]
\centering
\includegraphics[width=0.15\textwidth]{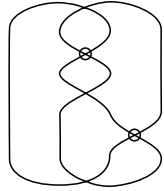}
\caption{Flat knot diagram}\label{fig:flat_knot_diagram}
\end{figure}

\begin{figure}[h]
\centering
\includegraphics[width=0.3\textwidth]{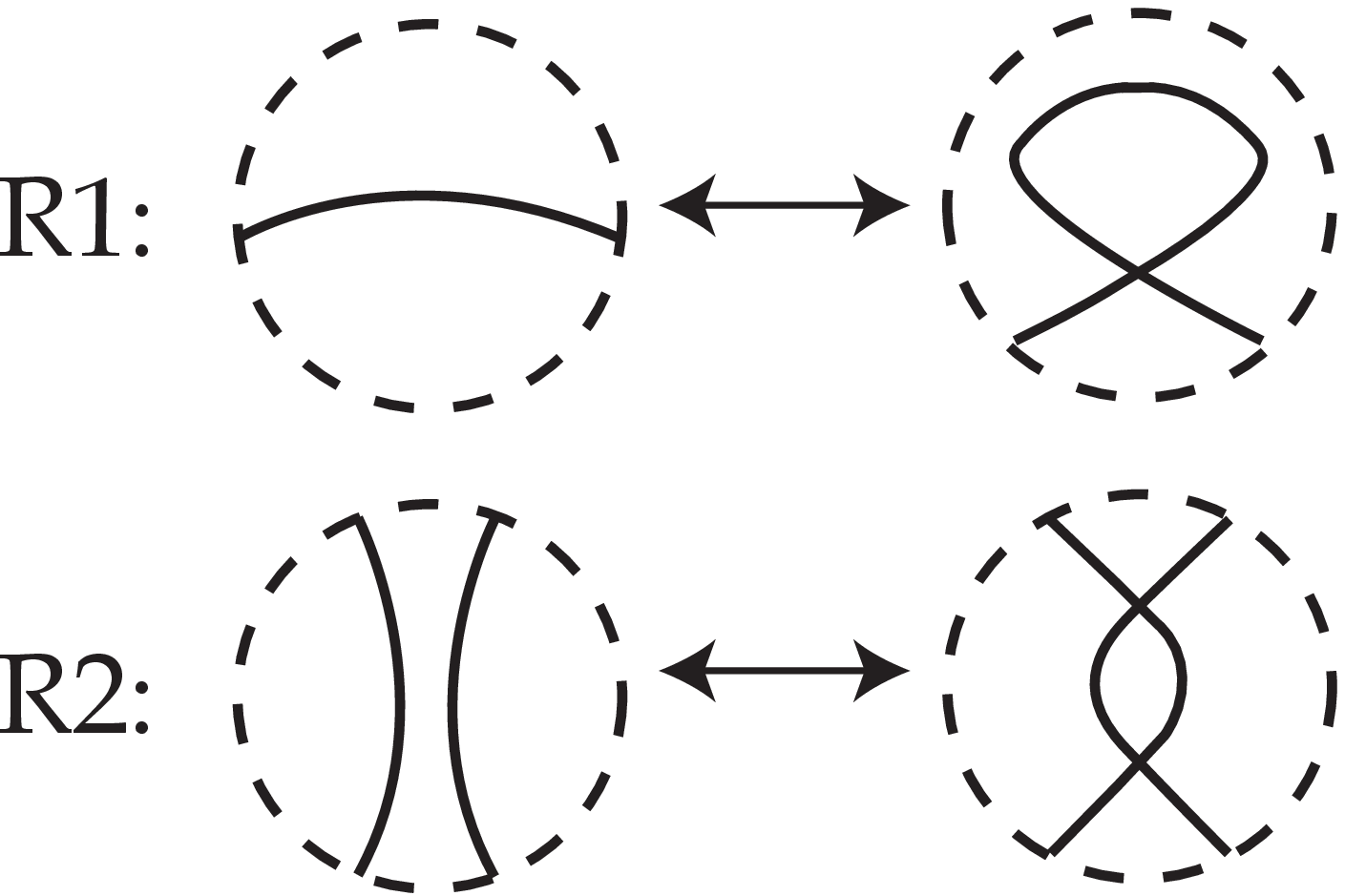}\qquad
\includegraphics[width=0.42\textwidth]{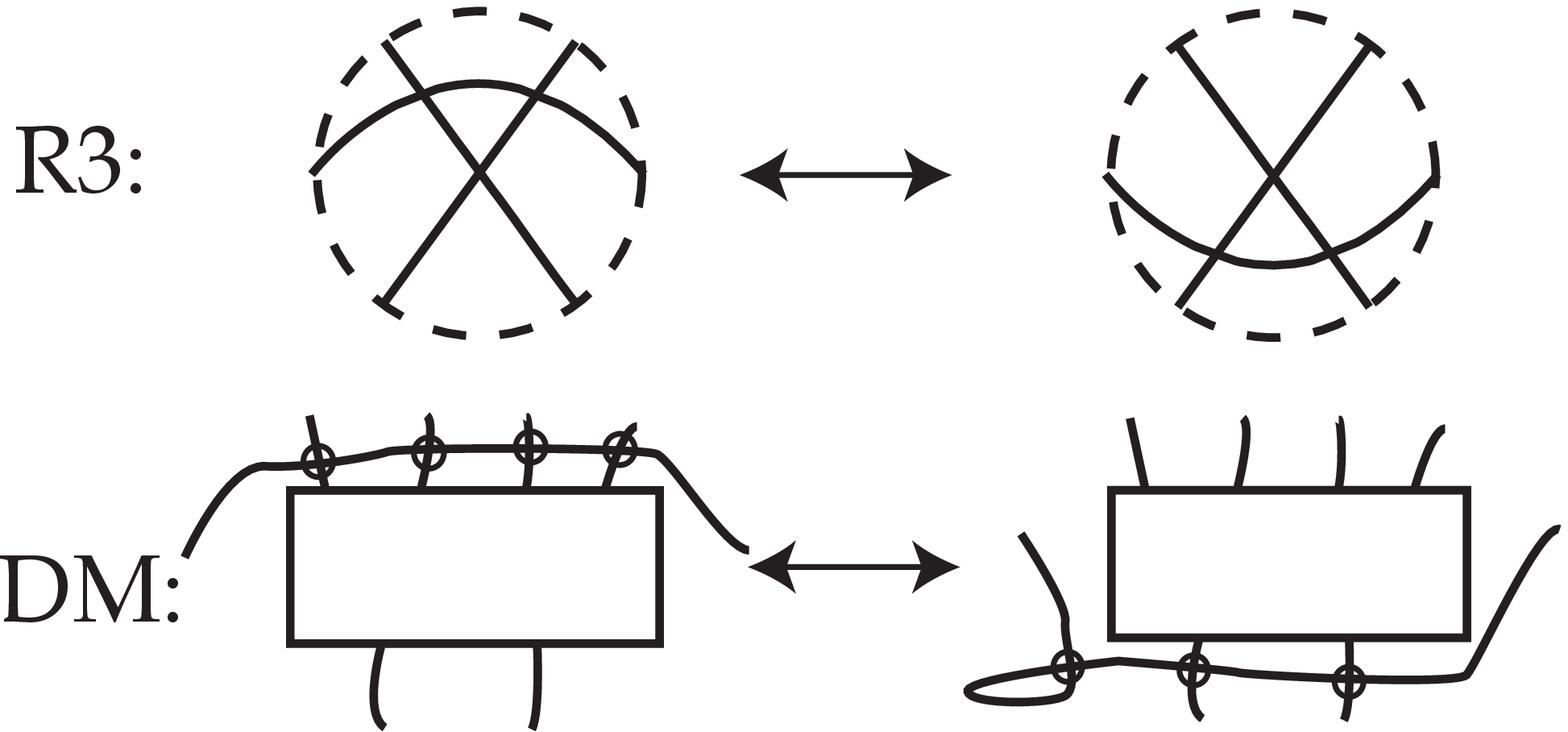}
\caption{Flat Reidemeister moves}\label{fig:flat_reidemeister_moves}
\end{figure}

%\begin{figure}
%\centering
%\includegraphics[width=0.25\textwidth]{flat_trefoil2.eps}
%\caption{Flat trefoil}\label{fig:flat_trefoil_surface}
%\end{figure}

Flat knots can be also identified with generic immersions of the circle into surfaces considered up to homotopies, isomorphisms and stabilizations/destabilizations.

\subsection{Based matrices}

Let $H$ be an abelian group.

\begin{definition}[\cite{T}]\label{def:based_matrix}
A {\em based matrix} with coefficients in the group $H$ is a triple $(G,s,b)$  where $G$ is a finite set, $s\in G$ and a map $b\colon G\times G\to H$ which defines a skew-symmetric matrix, i.e. $b(h,g)=-b(g,h)$ for any $g,h\in G$.
\end{definition}

Let us denote $G^\circ=G\setminus\{s\}$.

The motivating example of based matrices comes from intersection matrices for diagrams of flat knots. Let $H=\Z$ or $\Z_2$.

\begin{example}[Based matrix of a virtual (flat) knot]\label{exa:based_matrix_flat_knot}
Let $D$ be a diagram of an oriented virtual or flat knot in a surface $S$ and $\V(D)$ be the set of its classical crossings. For any crossing $c\in\V(D)$ define the left and right halves of the diagram at the crossing as shown in Fig.~\ref{fig:knot_halves}.

\begin{figure}[h]
\centering\includegraphics[width=0.5\textwidth]{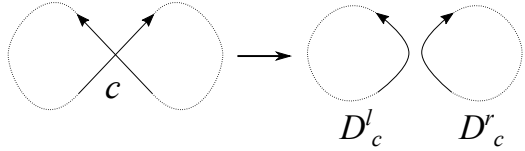}
\caption{Halves of the diagram}\label{fig:knot_halves}
\end{figure}

The halves $D^l_c$, $c\in\V(D)$, and the diagram $D$ correspond to elements of the homology group $H_1(S,H)$.

Let us define the based matrix of the diagram. Let $G(D)=\V(D)\sqcup\{s\}$ where $s$ is a formal element. The map $b_D$ is given by the formula

\begin{equation}\label{eq:based_matrix_flat_knot}
b_D(c,c')=D^l_c\cdot D^l_{c'},\quad b_D(c,s)=D^l_c\cdot D,\quad  b_D(s,s)=D\cdot D=0,
\end{equation}

where $c,c'\in\V(D)$ and the dot denotes the intersection map $H_1(S,H)\times H_1(S,H)\to H_0(S,H)=H$. In other word, $b_D(c,c')$ is the intersection number of the cycles $D^l_c$ and $D^l_{c'}$. The triple $T(D)=(G(D),s,b_D)$ is the {\em based matrix of the diagram} $D$.

For example, in case $H=\Z_2$ for the diagram $D$ in Fig.~\ref{fig:flat_knot_diagram} with three classical crossings we have $G(D)=\{s,1,2,3\}$ and
$$
b_D=\left(\begin{array}{c|ccc}
   0 & 1 & 1 & 0 \\
   \hline
   1 & 0 & 0 & 0 \\
   1 & 0 & 0 & 1 \\
   0 & 0 & 1 & 0
  \end{array}\right).
$$

\end{example}

\begin{definition}\label{def:based_matrix_elements}
Let $T=(G,s,b)$ be a based matrix. Then
\begin{itemize}
  \item element $g\in G^\circ$, is called {\em annihilating} if $b(g,h)=0$ for any $h\in G$;
  \item element $g\in G^\circ$, is called {\em core} if $b(g,h)=b(s,h)$ for any $h\in G$;
  \item elements $g_1,g_2\in G^\circ$, are called {\em complementary} if $b(g_1,h)+b(g_2,h)=b(s,h)$ for any $h\in G$.
\end{itemize}
\end{definition}

\begin{definition}\label{def:bm_homology}
Two based matrices are {\em homologous} if one can be obtained from the other by a sequence of the operations:
\begin{itemize}
\item $M_1$ which transforms $(G,s,b)$ into $(G_1=G\sqcup\{g\},s,b_1)$ such that $b_1\colon G_1\times G_1\to H$ extends $b$ and $b_1(g,h)=0$ for all $h\in G_1$;
\item $M_2$ which transforms $(G,s,b)$ into $(G_2=G\sqcup\{g\},s,b_1)$ such that $b_2\colon G_2\times G_2\to H$ extends $b$ and $b_2(g,h)=b_2(s,h)$ for all $h\in G_2$;
\item $M_3$ which transforms $(G,s,b)$ into $(G_1=G\sqcup\{g,g'\},s,b_1)$ such that $b_3\colon G_3\times G_3\to H$ is any skew-symmetric map extending $b$ with $b_3(g,h)+b_3(g',h)=b_3(s,h)$ for all $h\in G_3$
\end{itemize}
and the inverse operations $M_1^{-1},M_2^{-1},M_3^{-1}$.
\end{definition}

\begin{definition}\label{def:bm_primitive}
A based matrix is called {\em primitive} if it contains neither annihilating, nor core, nor complementary elements.
\end{definition}

\begin{lemma}[\cite{T}]\label{lem:unique_primitive_based_matrix}
Any based matrix is homologous to a unique up to isomorphism primitive based matrix.
\end{lemma}

\begin{lemma}[\cite{T}]\label{lem:funtoriality_based_matrix}
If two virtual (flat) diagrams are equivalent then the corresponding based matrices are homologous.
\end{lemma}

The Lemmas~\ref{lem:unique_primitive_based_matrix} and~\ref{lem:funtoriality_based_matrix} imply that the unique primitive based matrix $T_\bullet(D)=(G_\bullet(D),s,b_\bullet)$ homologous to the based matrix $T(D)=(G(D),x,b_D)$ of a virtual (flat) knot diagram is an invariant of virtual (flat) knots.

\subsection{Parity}

Let $\mathcal K$ be a virtual or flat knot. Consider the set $\mathfrak K$ of the diagrams of the knot $K$. The set $\mathfrak K$ can be considered as the objects of a diagram category whose morphisms are compositions of isotopies and Reidemeister moves.

For any diagrams $D,D'\in\mathfrak K$ and a morphism $f\colon D\to D'$ between them, there is a correspondence $f_*\colon\V(D)\to\V(D')$ between the crossings of the diagrams. The map $f_*$ is a partial bijection between the crossing set because some crossing can disappear within the transformation $f$ and other crossing can appear by a first or second Reidemeister move.

Let $A$ be an abelian group.

\begin{definition}
A {\em parity} on the diagrams of the knot $\mathcal K$ with coefficients in the group $A$ is a family of maps $p_D\colon \V(D)\to A$ from the sets of crossings $\V(D)$ of knot diagrams $D$,  such that

  \begin{itemize}
  \item Reidemeister moves do not change the parity value of any crossing (which survives under the move),
  \item the sum of parities values of the crossing participating in a Reidemeister move is equal to zero (Fig.~\ref{fig:parity_axioms}).
  \end{itemize}

\begin{figure}[h]
\centering
\includegraphics[width=0.7\textwidth]{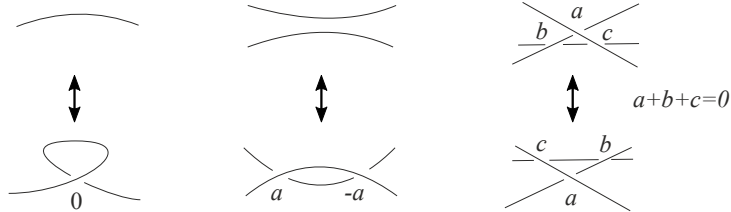}
\caption{Parity axioms}\label{fig:parity_axioms}
\end{figure}

\end{definition}

The first and the main example of parity on virtual knots is the Gaussian parity (with coefficients in $\Z_2$).

\begin{example}
The {\em Gaussian parity} of a crossing is the parity of the number of classical crossings that lie on a half of the knot corresponding to the crossing. For example, a diagram of a virtual eight-knot in Fig.~\ref{fig:gaussian_parity_example} has two odd and one even crossings.
\end{example}

\begin{figure}[h]
\centering
\includegraphics[height=3cm]{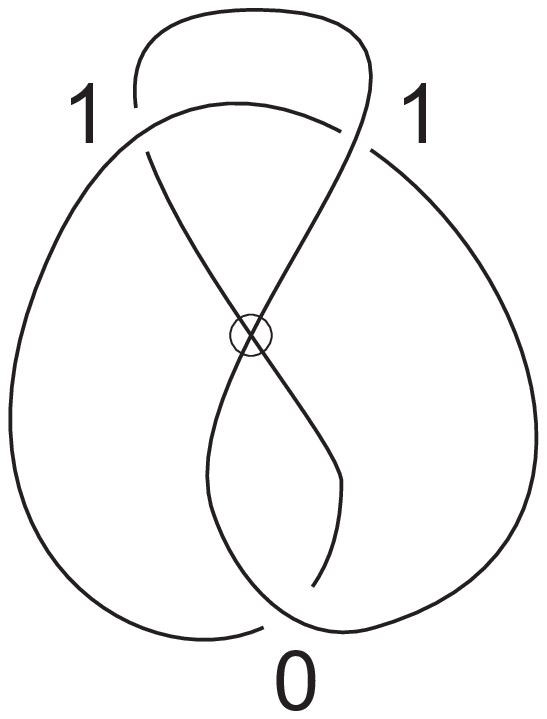}
\caption{Gaussian parity on a virtual knot diagram}\label{fig:gaussian_parity_example}
\end{figure}

%\begin{block}{Generalized link number}
%$lk = \frac 12 \sum_{c\mbox{{\scriptsize\ odd}}}sgn(c)$ \quad \includegraphics[height=1cm]{signs.eps}
%\end{block}

%\begin{definition}
%A {\em parity functor} is a family of maps $p_D\colon {\mathcal C}(D)\to A_D$ to abelian groups such that for any diagrams $D$ and $D'$ connected by a Reidemeister move or isomorphism $f$ a (partial) isomorphism $f_*\colon A_D\to A_D'$ is defined with properties
%  \begin{itemize}
%  \item $p_{D'}(f(c))=f_*(p_D(c))$ for any crossing $c$ which survives after the move $f$,
%  \item if $f$ is a Reidemeister move then $\sum_{c\mbox{\scriptsize\ participates in } f}p_D(c)=0$.
%  \end{itemize}
%\end{definition}

Another example are homological parities on diagrams of knots in a fixed surface $S$.

 \begin{theorem}[\cite{IMN}]\label{thm:homological_parity}
  Let $p$ be a parity for diagrams of some (flat) knot $\mathcal K$ in a surface $S$ with coefficients in a group $A$. Then there exists a unique homomorphism $\phi\colon H_1(S,\Z_2)/[K]\to A$ such that for any crossing $v$ of a knot diagram $D$ its parity is equal to $p(v)=\phi[D^l_v]$. Here $D^l_v$ is the left half of the diagram at the vertex $v$ (see Fig.~\ref{fig:knot_halves}).
\end{theorem}

As a corollary of the theorem we have the statement that any parity on classical knots is trivial.

Theorem~\ref{thm:homological_parity} shows that a parity on diagrams of knots in a fixed surface is determined by the homology of the surface. Since homological information on the knot is accumulated by the based matrix of the diagram, we can express the parity in terms of based matrices and use this expression to define a parity for virtual and flat knots.

\section{Parity on based matrices}\label{sect:parity_based_matrices}

Let $A$ be an abelian group.

\begin{definition}\label{def:parity_based_matrices}
A {\em parity on based matrices with coefficients in $A$} is a family of maps $p_T\colon G\to A$ defined for any based matrix $T=(G,s,b)$ such that
%\centering\includegraphics[height=2cm]{link1.png}
  \begin{itemize}
  \item[(P0)] if a based matrix $T'$ is obtained from a based matrix $T$ by adding an annihilating or core element of a pair of complementary elements then for any $g\in G$ $p_T(g) = p_{T'}(g)$;
  \item[(P1)] if $g\in G^\circ$ is an annihilating or core elements then $p_T(g)=0$,
  \item[(P2)] if $g_1,g_2\in G^\circ$ are complementary then $p_T(g_1)+p_T(g_2)=0$,
  \item[(P3)] if $g_1,g_2,g_3\in G^\circ$ such that $b(g_1,h)+b(g_2,h)+b(g_3,h)=b(s,h)$ for any $h\in G$ then $p_T(g_1)+p_T(g_2)+p_T(g_3)=0$.
  \end{itemize}
  By definition we set $p_T(s)=0$.
\end{definition}

\begin{example}
{\em Gaussian parity} with coefficients in $A=\Z$ or $\Z_2$ on based matrices with coefficients in $A$ is given by the formula $p_T(g)=b(g,s)\in A$. The parity properties follows from the fact that the moves  $M_i$, $i=1,2,3$, extend the form $b$ and in particular keep the value $b(g,s)$, $g\in G^\circ$, and from the equality $b(s,s)=0$.
\end{example}

Below we use a more general notion than parity --- parity functor~\cite{N}. Parity functor uses different coefficient groups for different diagrams, and the coefficient groups are linked by (partial) isomorphism which allow to identify parity values of a crossing under transformations (isotopies and Reidemeister moves). Here is the formal definition.

\begin{definition}\label{def:parity_functor_based_matrices}
A {\em parity functor on based matrices} is a family of maps $P_T\colon G\to A(T)$ into abelian groups $A(T)$ defined for any based matrix $T=(G,s,b)$ such that

  \begin{itemize}
  \item[(P0)] if a based matrix $T'=(G',s,b')$ is obtained from a based matrix $T=(G,s,b)$ by a move $f$ of type $M_1$, $M_2$ or $M_3$ then there is a fixed monomorphism $A(f)\colon A(T)\to A(T')$ such that $P_{T'}(g)=A(f)(P_T(g))$ for any $g\in G$;
  \item[(P1)] if $g\in G^\circ$ is an annihilating or core element, or $g=s$ then $p_T(g)=0$,
  \item[(P2)] if $g_1,g_2\in G^\circ$ are complementary then $P_T(g_1)+P_T(g_2)=0$,
  \item[(P3)] if $g_1,g_2,g_3\in G^\circ$ such that $b(g_1,h)+b(g_2,h)+b(g_3,h)=b(s,h)$ for any $h\in G$ then $P_T(g_1)+P_T(g_2)+P_T(g_3)=0$.
  \end{itemize}
\end{definition}

Any parity $p$ with coefficients in a group $A$ is a parity functor with the coefficient groups $A(T)=A$, the maps $P_T=p_T$ and the monomorphisms $A(f)=id_A$.

%\begin{proposition}
%Let $P$ be a parity functor on based matrices. Let $g_1,g_2\in G\setminus\{s\}$ be two elements such that $b(g_1,h)=b(g_2,h)$ for all $h\in G$. Then $P_M(g_1)=P_M(g_2)$.
%\end{proposition}

%\begin{proof}
%Add two complementary elements $g$ and $g'$ such that $b(g_1,h)=b(g,h)$ for all $h\in G$, and $b(g,g')=b(g_1,s)$. Then $b(g_2,h)=b(g,h)$ for all $h\in G$, $b(g,g_1)=b(g_1,g_1)=0$  and $b(g,g_2)=b(g_1,g_2)=b(g_2,g_2)=0$, and $b(g_1,g')=b(g_2,g')=b(g_1,s)$. Hence, for all $h\in G'=\cup\{g,g'\}$ we have $b(g_i,h)+b(g',h)=b(g,h)+b(g',h)=b(s,h)$, $i=1,2$. Thus, $g_1$ and $g'$ as well as $g_2$ and $g'$ are complementary. Then $P_{M'}(g_1)+P_{M'}(g')=P_{M'}(g_2)+P_{M'}(g')=0$. Hence, $P_{M'}(g_1)=P_{M'}(g_2)$,
%and $P_{M}(g_1)=P_{M}(g_2)$ by the property (P0).
%\end{proof}
%
%\begin{proposition}
%Let $P$ be a parity functor on based matrices. Let $g_1,\dots,g_k\in G\setminus\{s\}$ be elements such that $\sum_{i=1}^k b(g_i,h)=t\cdot b(s,h)$ for some $t\in\Z$ and any $h\in G$. Then $\sum_{i=1}^k P_M(g_i)=0$
%\end{proposition}
%
%\begin{proof}
%Firstly, let us prove that for any element $g\in G$ such that $b(g,h)=t\cdot b(s,h)$ for some $t\in\Z$ and any $h\in G$ its parity is zero: $P_M(g)=0$.
%For
%
%
%Let us prove  the statement for $k=3$. If $t=1$ then the proposition follows from the property (P3).
%\end{proof}

\begin{definition}
Let $T=(G,s,b)$ be a based matrix and $\mathfrak C$ be some partition of $G$ into disjoint subsets such that $\{s\}\in \mathfrak C$. The subgroup
\begin{equation}\label{eq:annulator}
Ann({\mathfrak C})= \{v\in\Z[G^\circ]\,|\, \exists k\in\Z :\, \forall C\in {\mathfrak C}\ b(v,\chi_C)=k\cdot b(s,\chi_C)\}
\end{equation}

in $\Z[G^\circ]$ where $\chi_C=\sum_{g\in C}g$, is called the {\em annulator} of the partition $\mathfrak C$.

The {\em derived partition} $A\mathfrak C$ of $\mathfrak C$ is the partition which consists of $\{s\}$ and the equivalence classes of the following relation on $G^\circ$:

\begin{equation}\label{eq:derived_partition}
g_1\simeq g_2 \Longleftrightarrow  g_1-g_2\mbox{ or }g_1+g_2\in Ann({\mathfrak C}).
\end{equation}
\end{definition}

We say that a partition $\mathfrak C$ of $G$ is {\em finer} than a partition $\mathfrak C'$  (${\mathfrak C}\succ {\mathfrak C}'$) if all classes of $\mathfrak C'$ are unions of some classes of $\mathfrak C$.

\begin{lemma}\label{lem:derivation_finer_partition}
If ${\mathfrak C}\succ {\mathfrak C}'$ then $A{\mathfrak C}\succ A{\mathfrak C}'$.
\end{lemma}

\begin{proof}
Let ${\mathfrak C}\succ {\mathfrak C}'$. Then $Ann(\mathfrak C)\subset Ann(\mathfrak C')$. Indeed, if $v\in Ann(\mathfrak C)$ then $b(v,\chi_C)=k\cdot b(s,\chi_C)$ for any $C\in \mathfrak C$. For any $C'\in \mathfrak C'$ we have $C'=\sqcup_i C_i, C_i\in \mathfrak C$. Then $\chi_{C'}=\sum_i\chi_C$ and
$$b(v,\chi_{C'})=\sum_i b(v,\chi_{C_i})=\sum_i k\cdot b(s,\chi_{C_i})=k\cdot b(s,\chi_{C'}).$$

For the relation $A{\mathfrak C}\succ A{\mathfrak C}'$ it suffices to show that for any $C\in A{\mathfrak C}$ and $C'\in A{\mathfrak C'}$ such that $C\cap C'\ne\emptyset$ one has inclusion $C\subset C'$. Let $x\in C\cap C'$ and $y\in C$. Then $x\pm y\in Ann(\mathfrak C)\subset Ann(\mathfrak C')$. Hence, $x$ and $y$ belong to the same class in $A{\mathfrak C}'$, i.e. $x,y\in C'$. Thus, $C\subset C'$.
\end{proof}

Let ${\mathfrak C}_0 = \bigcup_{g\in G}\{\{g\}\}$ and ${\mathfrak C}_n = A^n{\mathfrak C}_0$. Then ${\mathfrak C}_0\succ{\mathfrak C}_1\succ{\mathfrak C}_2\succ\dots$. Let ${\mathfrak C}_\infty = \lim_{n\to\infty} {\mathfrak C}_n$. The partition ${\mathfrak C}_\infty$ is called the {\em stable partition}. Let $U = Ann({\mathfrak C}_\infty)$ and $A^{st} = \Z[G^\circ]/U$. There is a natural map $P^{st}\colon G^\circ\to A^{st}$ which maps an element $g$ to the coset $gU$.

Repeating this construction for all based matrices $T=(G,s,b)$, we get a family of stable partitions ${\mathfrak C}_\infty(T)$, groups  $A^{st}(T) = \Z[G^\circ]/Ann({\mathfrak C}_\infty(T))$ and maps $P^{st}_T\colon G^\circ\to A^{st}(T)$. We claim that these maps form a parity functor on based matrices. We will check the claim in a more general situation.

\begin{definition}
A family $\mathfrak C$ of partitions $\mathfrak C(T)$ of the sets $G$ where $T=(G,s,b)$ are based matrices, is called a {\em tribal system} if
\begin{itemize}
\item for any $T=(G,s,b)$ $\{s\}\in \mathfrak C(T)$;
\item if a based matrix $T'$ is obtained from a based matrix $T=(G,s,b)$ by adding an annihilating or core element of a pair of complementary elements then $\mathfrak C(T)=\mathfrak C(T')\cap G$, i.e. for any $C\in\mathfrak C(T)$ there exists a unique $C'\in\mathfrak C(T')$ such that $C=C'\cap G$;
\item if $g_1,g_2\in G^\circ$ in a based matrix $T=(G,s,b)$ are complementary then there exists $C\in\mathfrak C(T)$ such that $g_1,g_2\in C$.
\end{itemize}

The classes $C\in\mathfrak C(T)$ are called {\em tribes}.
\end{definition}

\begin{remark}
We can reformulate the second and the third properties of tribal system as follows:
\begin{itemize}
  \item if two elements belong to one tribe and survive after a transformation of the base matrix then after the transformation they remain in one tribe;
  \item complementary vectors belong to one tribe.
\end{itemize}
\end{remark}

\begin{theorem}\label{thm:tribe_system_parity_functor}
  Let $\mathfrak C$ be a tribal system. For any based matrix $T=(G,s,b)$, let $A^{\mathfrak C}(T) = \Z[G^\circ]/Ann({\mathfrak C}(T))$ and $P^{\mathfrak C}_T\colon G^\circ\to A^{\mathfrak C}(T)$ be the natural maps. Then the maps $P^{\mathfrak C}_T$ define a parity functor on based matrices with coefficients $A^{\mathfrak C}(T)$.
\end{theorem}

The parity functor $P^{\mathfrak C}_T$ is called the {\em parity functor associated with the tribal system $\mathfrak C$}.

First we prove an auxiliary lemma.

\begin{lemma}\label{lem:annulator_P0}
Let $T=(G,s,b)$ be a based matrix and $T'=(G',s,b')$ be the result of an operation $M_1$, $M_2$ or $M_3$ applied to $T$. Let $\mathfrak C'$ be a partition of $G'$ such that $\{s\}\in\mathfrak C'$ and there exists $C'\in\mathfrak C'$ such that $G'\setminus G\subset C'$. Then
\begin{enumerate}
  \item $Ann(\mathfrak C'\cap G)=Ann(\mathfrak C')\cap \Z[G^\circ]$
  \item $A(\mathfrak C'\cap G)=A(\mathfrak C')\cap G$
\end{enumerate}
\end{lemma}

\begin{proof}
  1)\ Since $G'\setminus G\subset C'$, we have $\mathfrak C'\cap G=\mathfrak C'\setminus\{C'\}\cup\{C\}$ where $C=C'\cap G$. By definition we have $b'(v,\chi_{C'})=b'(v,\chi_{C})+l\cdot b'(v,s)$ for any $v\in\Z[G'^\circ]$ where $l=0$ when the element in $G'\setminus G$ is annihilating, and $l=1$ when the new element is core or when $G'\setminus G$ consists of a pair of complementary elements. Hence, $b'(s,\chi_{C'})=b'(s,\chi_{C})$.  Then the pair of conditions for the classes $C'$ and $\{s\}$ in the definition of $Ann(\mathfrak C')$ is eqivalent to the pair of conditions for the classes $C$ and $\{s\}$:
  $$
  \left\{\begin{array}{c}
           b'(v,\chi_{C'})=k\cdot b'(s,\chi_{C'}) \\
           b'(v,s)=k\cdot b'(s,s)=0
         \end{array}\right. \Longleftrightarrow
         \left\{\begin{array}{c}
           b'(v,\chi_{C})=k\cdot b'(s,\chi_C) \\
           b'(v,s)=0
         \end{array}\right.
  $$
  This means that we can replace the tribe $C'$ with $C$ in the definition of $Ann({\mathfrak C}(T'))$, i.e. replace the partition $\mathfrak C'$ with the partition $\mathfrak C'\setminus\{C'\}\cup\{C\}=\mathfrak C'\cap G$. Thus,
    \begin{multline*}
    Ann({\mathfrak C'})\cap\Z[G^\circ]=\\ \{v\in\Z[G^\circ]\,|\, \exists k\in\Z :\, \forall C''\in {\mathfrak C'}\ b'(v,\chi_{C''})=k\cdot b'(s,\chi_{C''})\}=\\
    \{v\in\Z[G^\circ]\,|\, \exists k\in\Z :\, \forall C''\in {\mathfrak C'\cap G}\ \ b(v,\chi_{C''})=k\cdot b(s,\chi_{C''})\}=Ann({\mathfrak C'}\cap G).
  \end{multline*}

  2)\ Let $\sim_{A(\mathfrak C'\cap G)}$ and $\sim_{A(\mathfrak C')}$ be the equivalence relations on $G^\circ$ determined by the correspondent partitions. Then for any $g_1,g_2\in G^\circ$ we have a chain of equivalences
  $$
  g_1\sim_{A(\mathfrak C'\cap G)} g_2 \Leftrightarrow g_1\pm g_2\in Ann({\mathfrak C'}\cap G) \Leftrightarrow g_1\pm g_2\in Ann({\mathfrak C'}) \Leftrightarrow g_1\sim_{A(\mathfrak C')} g_2.
  $$
  Thus, the partitions $A(\mathfrak C'\cap G)$ and $A(\mathfrak C')\cap G$ coincide.
\end{proof}

\begin{proof}[Proof of Theorem~\ref{thm:tribe_system_parity_functor}]
  The properties (P1),(P2),(P3) follow from the definition of the annulator $Ann({\mathfrak C}(T))$.

  Let us check the property (P0). Let a based matrix $T'=(G',s,b')$ is obtained from a based matrix $T=(G,s,b)$ by adding an annihilating or core element of a pair of complementary elements. We need to show that the natural map $A^{\mathfrak C}(T)\to A^{\mathfrak C}(T')$ which maps $g\in G$ to $g$, is a well defined monomorphism. This condition is equivalent to the equality $Ann({\mathfrak C}(T'))\cap\Z[G^\circ]=Ann({\mathfrak C}(T))$. The last statement holds by Lemma~\ref{lem:annulator_P0}.

\end{proof}

\begin{proposition}\label{prop:stable_tribal_system}
The stable partition $\mathfrak C_\infty$ is a tribal system.
\end{proposition}

\begin{proof}
Let $T=(G,s,b)$ be a based matrix. By the definition of annulator, for any partition $\mathfrak C(T)$ of $G$ any pair of complementary elements belongs to one tribe of the partition $A\mathfrak C(T)$.

Let $T'=(G',s,b')$ is obtained from $T$ by an operation $M_1$, $M_2$ or $M_3$. Let us show that $\mathfrak C_\infty(T')\cap G=\mathfrak C_\infty(T)$.

Let $\tilde{\mathfrak C}_0(T')$ be the partition of $G'$ which consists of $\{s\}$, one-element sets $\{g\}$, $g\in G^\circ$, and the subset $G'\setminus G$. Then $\mathfrak C_0(T')\succ \tilde{\mathfrak C}_0(T') \succ A\mathfrak C_0(T')=\mathfrak C_1(T')$.

Let $\tilde{\mathfrak C}_n(T')=A^n\tilde{\mathfrak C}_0(T')$ and $\tilde{\mathfrak C}_\infty(T')=\lim_{n\to\infty}\tilde{\mathfrak C}_n(T')$. By Lemma~\ref{lem:derivation_finer_partition} we have $\mathfrak C_n(T')\succ \tilde{\mathfrak C}_n(T') \succ \mathfrak C_{n+1}(T')$. Hence, $\tilde{\mathfrak C}_\infty(T')=\mathfrak C_\infty(T')$.

By definition $\tilde{\mathfrak C}_0(T')\cap G=\mathfrak C_0(T)$. Lemma~\ref{lem:annulator_P0} implies that
$$\tilde{\mathfrak C}_n(T')\cap G=A^n\tilde{\mathfrak C}_0(T')\cap G=A^n(\tilde{\mathfrak C}_0(T')\cap G)=A^n\mathfrak C_0(T)=\mathfrak C_n(T).$$

Then $\mathfrak C_\infty(T)=\tilde{\mathfrak C}_\infty(T')\cap G=\mathfrak C_\infty(T')\cap G$. Thus, $\mathfrak C_\infty$ is a tribal system.
\end{proof}

Theorem~\ref{thm:tribe_system_parity_functor} and Proposition~\ref{prop:stable_tribal_system} imply the following corollary

\begin{corollary}\label{cor:stable_parity_functor}
The family of maps $P^{st}_T$ is a parity functor (called the {\em stable parity functor}) with coefficients in the groups $A^{st}(T)$.
\end{corollary}

\begin{example}\label{exa:stable_parity_functor}

Consider a based matrix $T=(G,s,b)$ with coefficients in $\Z_2$ where $G=\{s,1,2,3,4,5,6,7,8\}$ and $b$ be given by induces the matrix $B$ ($s$ corresponds to the first row and column).

$$
B=
\begin{blockarray}{*{9}{c}}
 s & \mathbf{1} & \mathbf{2} & \mathbf{3} & \mathbf{4} & \mathbf{5} & \mathbf{6} & \mathbf{7} & \mathbf{8}\\
\begin{block}{(c|c|c|c|c|c|c|c|c)} %c|c|c|c|c|c|c|c|c
 0 & 0 & 0 & 0 & 0 & 0 & 0 & 0 & 0\\
\BAhline
 0 &  0 & 0 &  1 & 0 &  1 & 1 &  0 & 0\\
\BAhline
 0 &  0 & 0 &  1 & 0 &  1 & 1 &  0 & 0\\
\BAhline
 0 &  1 & 1 &  0 & 0 &  0 & 1 &  1 & 0\\
\BAhline
 0 &  0 & 0 &  0 & 0 &  0 & 1 &  1 & 0\\
\BAhline
 0 &  1 & 1 &  0 & 0 &  0 & 0 &  0 & 1\\
\BAhline
 0 &  1 & 1 &  1 & 1 &  0 & 0 &  0 & 1\\
\BAhline
 0 &  0 & 0 &  1 & 1 &  0 & 0 &  0 & 1\\
\BAhline
 0 &  0 & 0 &  0 & 0 &  1 & 1 &  1 & 0\\
\end{block}
\end{blockarray}
$$

Let us find the stable partition. We look for the pairs of rows whose sum is equal to the first row. Those are the rows $1$ and $2$. This means the elements $\{1,2\}$ form one class in the partition $\mathfrak C_1$. We replace the columns with numbers $1$ and $2$ by their sum (it corresponds to the vector $\chi_{\{1,2\}}$) and get the matrix $B_1$. We use separators between rows to distinguish the classes of the partition. Then we look for rows in the matrix $B_1$ as before. Repeating the process, we get the sequence of the partitions $\mathfrak C_2,\mathfrak C_3$ and the corresponding matrices $B_2, B_3$ whose rows correspond to the elements of $G$ and the columns correspond to the classes of the partition.

$$
B_1=
\left(
\begin{array}{c|c|c|c|c|c|c|c}
 0 &   0 &  0 & 0 &  0 & 0 &  0 & 0\\
 \hline
 0 &   0 &  1 & 0 &  1 & 1 &  0 & 0\\
 0 &   0 &  1 & 0 &  1 & 1 &  0 & 0\\
 \hline
 0 &   0 &  0 & 0 &  0 & 1 &  1 & 0\\
 \hline
 0 &   0 &  0 & 0 &  0 & 1 &  1 & 0\\
 \hline
 0 &   0 &  0 & 0 &  0 & 0 &  0 & 1\\
 \hline
 0 &   0 &  1 & 1 &  0 & 0 &  0 & 1\\
 \hline
 0 &   0 &  1 & 1 &  0 & 0 &  0 & 1\\
 \hline
 0 &   0 &  0 & 0 &  1 & 1 &  1 & 0
\end{array}
\right),
\qquad
B_2=
\left(
\begin{array}{c|c|c|c|c|c}
 0 &   0 &  0 &  0 & 0  & 0\\
 \hline
 0 &   0 &  1 &  1 & 1  & 0\\
 0 &   0 &  1 &  1 & 1  & 0\\
 \hline
 0 &   0 &  0 &  0 & 0  & 0\\
 0 &   0 &  0 &  0 & 0  & 0\\
 \hline
 0 &   0 &  0 &  0 & 0  & 1\\
 \hline
 0 &   0 &  0 &  0 & 0  & 1\\
 0 &   0 &  0 &  0 & 0  & 1\\
 \hline
 0 &   0 &  0 &  1 & 0  & 0
\end{array}
\right),
$$

$$
B_3=
\left(
\begin{array}{c|c|c|c|c}
 0 &   0 &  0 &  0  & 0\\
 \hline
 0 &   0 &  1 &  0  & 0\\
 0 &   0 &  1 &  0  & 0\\
 \hline
 0 &   0 &  0 &  0  & 0\\
 0 &   0 &  0 &  0  & 0\\
 \hline
 0 &   0 &  0 &  0  & 1\\
 0 &   0 &  0 &  0  & 1\\
 0 &   0 &  0 &  0  & 1\\
 \hline
 0 &   0 &  0 &  1  & 0
\end{array}
\right).
$$

\begin{gather*}
{\mathfrak C}_0 = \{\{s\},\{1\},\{2\},\{3\},\{4\},\{5\},\{6\},\{7\},\{8\}\},\\
{\mathfrak C}_1 = \{\{s\},\{1,2\},\{3\},\{4\},\{5\},\{6\},\{7\},\{8\}\},\\
{\mathfrak C}_2 = \{\{s\},\{1,2\},\{3,4\},\{5\},\{6, 7\},\{8\}\},\\
{\mathfrak C}_3 = \{\{s\},\{1,2\},\{3,4\},\{5,6,7\},\{8\}\}= {\mathfrak C}_\infty.
\end{gather*}

The partition $\mathfrak C_3$ stabilises (there are no two rows from different classes whose sum gives the first row of the table). Thus, $\mathfrak C_3=\mathfrak C_\infty$ is the stable partition.

The annulator of $\mathfrak C_\infty$ in $\Z[\{1,2,3,4,5,6,7,8\}]$ is generated by the elements $2\cdot\mathbf{i}$, $i=1,\dots,8$ and $\mathbf{1}+\mathbf{2}$, $\mathbf{3}$, $\mathbf{4}$, $\mathbf{5}+\mathbf{6}$, $\mathbf{5}+\mathbf{7}$ (we look for combinations of rows in the table $B_3$ that give zero row). Then $A^{st}=\Z_2\oplus\Z_2\oplus\Z_2$ and the parity map is the following
\begin{gather*}
P^{st}(\mathbf{1})=P^{st}(\mathbf{2})=(1,0,0),\quad P^{st}(\mathbf{3})=P^{st}(\mathbf{4})=(0,0,0),\\ P^{st}(\mathbf{5})=P^{st}(\mathbf{6})=P^{st}(\mathbf{7})=(0,1,0),\quad P^{st}(\mathbf{8})=(0,0,1).
\end{gather*}
\end{example}

Let us modify the definition of parity functor in Theorem~\ref{thm:tribe_system_parity_functor} in order to get a parity on based matrix.

\subsection{Reduced parity functor}

Let $T_0=(G_0,s_0,b_0)$ be a based matrix with coefficients in $H$ and let $\mathfrak C$ be a tribal system defined on the based matrices $T=(G,s,b)$ homologous to $T_0$.

For a based matrix $T=(G,s,b)$, consider the map $\hat P^{\mathfrak C}_T$ from $\Z(G^\circ)$ to the group $\hat A^{\mathfrak C}(T)=H[\mathfrak C(T)]/\langle \hat b(s)\rangle$ where $\langle \hat b(s)\rangle$ is the cyclic subgroup generated by
\begin{equation}\label{eq:bs_element}
\hat b(s)=\sum_{C\in\mathfrak C(T)} b(s,\chi_C)\cdot C,
\end{equation}
defined by the formula
\begin{equation}\label{eq:extended_parity_functor}
\hat P^{\mathfrak C}_T(g)=\sum_{C\in\mathfrak C(T)} b(g,\chi_C)\cdot C.
\end{equation}

The kernel of the homomorphism $\hat P^{\mathfrak C}_T$ is the annulator $Ann(\mathfrak C(T))$. Hence, the group $A^{\mathfrak C}(T)$ embeds in $\hat A^{\mathfrak C}(T)$. The family of maps $\hat P^{\mathfrak C}_T$ is a parity functor with coefficients in the groups $\hat A^{\mathfrak C}(T)$. For an operation $f\colon T\to T'$, $T=(G,s,b)$, $T'=(G',s,b')$, of type $M_1, M_2$ or $M_3$ the monomorphism $\hat A^{\mathfrak C}(f)\colon \hat A^{\mathfrak C}(T)\to \hat A^{\mathfrak C}(T')$ is given by the formula

\begin{equation}\label{eq:extended_parity_functor_coefficients}
\hat A^{\mathfrak C}(f)(\sum_{C\in\mathfrak C(T)}\lambda_C\cdot C)=\sum_{C\in\mathfrak C(T)}\lambda_C\cdot f(C)+l\cdot\lambda_s\cdot C'_f
\end{equation}

where $f(C)\in\mathfrak C(T')$ is the unique tribe such that $f(C)\cap G=C$, $\lambda_s=\lambda_{\{s\}}$, $C'_f\in\mathfrak C(T')$  is the tribe containing $G'\setminus G$, and $l=0$ when $f$ is of type $M_1$ and $l=1$ when $f$ is of type $M_2$ or $M_3$.

Let $T_\bullet=(G_\bullet, s, b_\bullet)$ be a primitive based matrix which is obtained from $T$ by deleting annihilating, core and complementary elements. Then $G_\bullet\subset G$.

\begin{definition}
A tribe $C\in\mathfrak C(T)$ is called {\em primitive} if $C\cap G_\bullet\ne\emptyset$. Denote the subset of primitive tribes in $\mathfrak C(T)$ by $\mathfrak C(T)_{prim}$.
\end{definition}

Note that by definition $\{s\}\in \mathfrak C(T)_{prim}$.

\begin{proposition}\label{prop:primitive_tribe_invariance}
1. The primitive tribes do not depend on the choice of the primitive based matrix $T_\bullet$, that is if $C$ is primitive with respect to a primitive based matrix $T_\bullet=(G_\bullet, s, b_\bullet)$ and $T'_\bullet=(G'_\bullet, s, b'_\bullet)$ is another primitive based matrix such that $G'_\bullet\subset G$ then $C\cap G'_\bullet\ne\emptyset$.

2. Let $T'=(G',s,b')$ be obtained from $T$ by an operation $M_1$, $M_2$ or $M_3$. Then there is a bijection between the primitive tribes in $\mathfrak C(T)$ and $\mathfrak C(T')$: a tribe $C'\in\mathfrak C(T')_{prim}$ iff the tribe $C'\cap G\in \mathfrak C(T)_{prim}$.

3. The support of the element $\hat b(s)\in H[\mathfrak C(T)]$ includes only primitive tribes: $\hat b(s)=\sum_{C\in\mathfrak C(T)_{prim}} b(s,\chi_C)\cdot C$
\end{proposition}

\begin{proof}
1. Let $T\stackrel{f_1}{\longrightarrow} T_1 \stackrel{f_2}{\longrightarrow}\cdots\stackrel{f_n}{\longrightarrow}T_\bullet$ and
$T\stackrel{f'_1}{\longrightarrow} T'_1 \stackrel{f'_2}{\longrightarrow}\cdots\stackrel{f'_{n'}}{\longrightarrow}T'_\bullet$ be sequences of operations $M_k^{-1}$, $k=1,2,3$, that reduce $T$ to the primitive based matrices. We prove the first statement of the proposition by induction on $d=|G|-|G_\bullet|$. Note that this number does not depend on the choice of primitive matrix because all homologous primitive matrix are isomorphic and have the same number of elements.

If $d=0$ then $T_\bullet=T=T'_\bullet$, and the statement is trivial.

Assume the statement holds for all $k<d$. Consider the operations $f_1$ and $f'_1$. If $f_1=f'_1$ consider the based matrix $T''=f_1(T)$, $T''=(G'',s,b'')$. Then $G_\bullet\subset G''$ and $G'_\bullet\subset G''$. If $C\in\mathfrak C(T)$ then $C\cap G_\bullet\ne\emptyset$. Hence, $(C\cap G'')\cap G_\bullet\ne\emptyset$, i.e. the tribe $C''=C\cap G''$ is primitive in $\mathfrak C(T'')=\mathfrak C(T)\cap G''$. By induction $C''\cap G'_\bullet\ne\emptyset$. Then $C\cap G'_\bullet\ne\emptyset$.

Let $f_1\ne f'_1$ commute. Let $T_1=f_1(T)$, $T'_1=f'_1(T)$ and $T''=f'_1(T_1)=f_1(T'_1)$. Then $G_\bullet\subset G_1$ and $G'_\bullet\subset G'_1$.  By operations $M_k^{-1}$, $k=1,2,3$, the based matrix $T''$ can be reduced to a primitive based matrix $T''_\bullet=(G''_\bullet,s, b''_\bullet)$. Then $G''_\bullet\subset G''=G_1\cap G'_1$.

Let $C\in\mathfrak C(T)$ be primitive, i.e. $C\cap G_\bullet\ne\emptyset$.  Let $C_1=C\cap G_1\in\mathfrak C(T_1)$ and $C'_1=C\cap G'_1\in\mathfrak C(T'_1)$. By induction, $C_1\cap G_\bullet=C\cap G_\bullet\ne\emptyset$ implies $C_1\cap G''_\bullet\ne\emptyset$. But
$C_1\cap G''_\bullet=C\cap G''_\bullet=C'_1\cap G''_\bullet$. Then $C'_1\cap G''_\bullet\ne\emptyset$. Hence, by induction $C\cap G'_\bullet=C'_1\cap G'_\bullet\ne\emptyset$.

Let $f_1\ne f'_1$ do not commute. If one of the operations (say, $f_1$) is of type $M_3^{-1}$ and the other (say, $f'_1$) is of type $M_1^{-1}$ or $M_2^{-1}$ then $f_1$ removes a pair of complementary core and annihilating elements. Then $f_1=f''\circ f'_1$ for some operation $f''$ of type $M_1^{-1}$ or $M_2^{-1}$, and we can treat this case like the case $f_1=f'_1$ by considering the based matrix $T''=f'_1(T)$.

The other option is when $f_1$ and $f'_1$ are of type $M_3^{-1}$. Denote $T_1=f_1(T)$ and $T'_1=f'_1(T)$. The sets $G\setminus G_1=\{g_0,g'_1\}$ and $G\setminus G'_1=\{g_0,g_1\}$ are intersecting pairs of complementary elements. Then $b(g,g_1)=b(g,g'_1)$ for all $g\in G$, and the map $h\colon G_1\to G'_1$, which replaces the element $g_1\in G_1$ with the element $g'_1\in G'_1$, is an isomorphism of the based matrices $T_1$ and $T'_1$. Hence, the based matrix $T''_\bullet=h(T_\bullet)$ is a primitive submatrix of $T'_1$ which is homologous to $T'_1$.

Let $C\in\mathfrak C(T)$ be primitive, i.e. $C\cap G_\bullet\ne\emptyset$. If there exists $g''\ne g_1$ such that $g''\in C\cap G_\bullet$ then $g''\in C \cap G''_\bullet = C'_1 \cap G''_\bullet$ where $C'_1=C\cap G'_1$. Then by induction $C\cap G'_\bullet=C'_1 \cap G'_\bullet\ne\emptyset$.

Assume that $C\cap G_\bullet=\{g_1\}$. Since $g_0$ and $g_1$ as well as $g_0$ and $g'_1$ are complementary in $T$, they all belong to one tribe, i.e. $g_0,g_1,g'_1\in C$. Then $g'_1\in C\cap G''_\bullet$, so $C'_1 \cap G''_\bullet\ne\emptyset$ where $C'_1=C\cap G'_1$. Then by induction $C\cap G'_\bullet=C'_1 \cap G'_\bullet\ne\emptyset$.

2. Let $T_\bullet=(G_\bullet, s, b_\bullet)$ be a primitive based matrix obtained from $T$ by operations $M_i^{-1}, i=1,2,3$. Then $G_\bullet\subset G$. Then we have the equivalences
$$
C'\in\mathfrak(T')_{prim} \Longleftrightarrow C'\cap G_\bullet = (C'\cap G)\cap G_\bullet\ne\emptyset \Longleftrightarrow (C'\cap G)\in\mathfrak(T)_{prim}.
$$

3. By the proof of Lemma~\ref{lem:annulator_P0}, for any reducing operation $f\colon T\to T'$ of type $M_i^{-1}$, $i=1,2,3$, and any tribe $C\in\mathfrak C(T)$ we have $b(s,\chi_{C\cap G'})=b(s,\chi_C)$ where $T=(G,s,b)$ and $T'=(G',s,b|_{G'})$. If $C\in\mathfrak C(T)$ is not primitive then it reduces to an empty tribe on the way to a primitive based matrix. Then $b(s,\chi_C)=b(s,\chi_\emptyset)=0$. Thus,
$$\hat b(s)=\sum_{C\in\mathfrak C(T)} b(s,\chi_C)\cdot C=\sum_{C\in\mathfrak C(T)_{prim}} b(s,\chi_C)\cdot C.$$
\end{proof}

Thus, for any based matrix $T$ a distinguished set $\mathfrak(T)_{prim}$ of the primitive tribes is defined, and any transformation $f\colon T\to T'$ to a homologous based matrix $T'$ identifies the sets $\mathfrak(T)_{prim}$ and $\mathfrak(T')_{prim}$.

Let us consider another distinguished tribe in $\mathfrak C(T)$.

\begin{proposition}
Let $T_1=(G_1,s,b_1)$ and $T_2=(G_2,s,b_2)$ be homologous to $T$ and $g_i\in G_i\cap G$, $i=1,2$ be an element which is annihilating or core in $G_i$. Then $g_1$ and $g_2$ belong to one tribe in $\mathfrak C(T)$.
\end{proposition}

\begin{proof}
  Let $f_i\colon T_i\to T$ be the sequence of operations that transform $T_i$ to $T$. Add $T$, $T_1$, $T_2$  and all the intermediate diagrams an element $g'$ which is annihilating for all the diagrams $\tilde T$, $\tilde T_1$, $\tilde T_2$ and the intermediate ones. Then $g_i$ and $g'$ belong to one tribe in $\mathfrak C(T_i)$, $i=1,2$. Hence, they belong to one tribe in $\mathfrak C(T)$. Then by transitivity $g_1$ and $g_2$ belong to one tribe in $\mathfrak C(T)$.
\end{proof}

Let $C_0((T)\in\mathfrak C(T)$ be the tribe which contains all core and annihilating elements. We call it the {\em zero tribe}.

The zero tribe $C_0(T)$ can be empty. On the other hand, $C_0(T)$ can be primitive since it may include other elements besides the core and annihilating ones.

\begin{corollary}\label{cor:zero_tribe_invariance}
Let $T'=(G',s,b')$ be obtained from $T$ by an operation $M_1$, $M_2$ or $M_3$ and $C_0=C_0(T)$ and $C'_0=C_0(T')$ be the zero tribes in $\mathfrak C(T)$ and $\mathfrak C(T')$ correspondingly. Then $C_0=C'_0\cap G$.
\end{corollary}

Corollary~\ref{cor:zero_tribe_invariance} means the zero tribe is equivariant, that any transformation $f\colon T\to T'$ from $T$ to a homologous based matrix $T'$ maps elements of the zero tribes $\mathfrak(T)_{prim}$ to the zero tribe $\mathfrak(T')_{prim}$.

Consider a based matrix $T=(G,s,b)$.  Denote $$\mathfrak C(T)_{prim}^*=\mathfrak C(T)_{prim}\setminus\{C_0(T)\}.$$ Let
$\tilde A^{\mathfrak C}(T)=H[\mathfrak C(T)_{prim}^*]/\langle\tilde b(s)\rangle$
where
\begin{equation}\label{eq:tilde_bs_element}
\tilde b(s)=\sum_{C\in\mathfrak C(T)_{prim}^*} b(s,\chi_C)\cdot C.
\end{equation}

Define the maps
$\tilde P^{\mathfrak C}_T\colon G\to \tilde A^{\mathfrak C}(T)$ by the formula
\begin{equation}\label{eq:reduced_parity_functor}
\tilde P^{\mathfrak C}_T(g)= \sum_{C\in\mathfrak C(T)_{prim}^*} \left(b(g,\chi_C)-\left\lfloor\frac{|C|}2\right\rfloor\cdot b(g,s)\right)\cdot C
\end{equation}
where $\lfloor\cdot\rfloor$ is the floor function.

For an operation $f\colon T\to T'$, $T=(G,s,b)$, $T'=(G',s,b')$, of type $M_1, M_2$ or $M_3$  we define a homomorphism $\tilde A^{\mathfrak C}(f)\colon \tilde A^{\mathfrak C}(T)\to \tilde A^{\mathfrak C}(T')$ by the formula
\begin{equation}\label{eq:reduced_parity_functor_coefficients}
\tilde A^{\mathfrak C}(f)\left(\sum_{C\in\mathfrak C(T)_{prim}^*}\lambda_C\cdot C\right)=\sum_{C\in\mathfrak C(T)_{prim}^*}\lambda_C\cdot f(C)
\end{equation}
where $f(C)\in \mathfrak C(T')_{prim}$ is the unique primitive tribe such that $f(C)\cap G = C$.

\begin{theorem}\label{thm:reduced_parity_functor_based_matrices}
 The maps $\tilde P^{\mathfrak C}_T$ define a parity functor on based matrices.
\end{theorem}

The parity functor is called the {\em reduced parity functor associated with the tribal system $\mathfrak C$}.

\begin{proof}
  Since $\tilde P^{\mathfrak C}_T$ factors through $\hat P^{\mathfrak C}_T$, the conditions (P1),(P2) and (P3) are valid.

  Let us check the property (P0). Let $T=(G,s,b)$ be a based matrix and $T'=f(T)$  where $f$ is of type $M_i$, $i=1,2,3$. Let $T'=(G',s,b')$.

  First we show that $\tilde A^{\mathfrak C}(f)$ is a well defined isomorphism. By Proposition~\ref{prop:primitive_tribe_invariance} the map $f$ establishes a bijection  $f\colon \mathfrak C(T)_{prim}\to \mathfrak C(T')_{prim}$. By Corollary~\ref{cor:zero_tribe_invariance} $f(C_0(T))=C_0(T')$ if $C_0(T)\in \mathfrak C(T)_{prim}$. Then $f$ defines an isomorphism $f\colon H[\mathfrak C(T)_{prim}^*]\to H[\mathfrak C(T')_{prim}^*]$.

  For any tribe $C'\in\mathfrak C(T')$ we have $b'(s,\chi_C)=b(s,\chi_{C'\cap G})$. Then $f(\tilde b(s))=\tilde b(s)$. Hence, the isomorphism $f$ induces an isomorphism $\tilde A^{\mathfrak C}(T)\to \tilde A^{\mathfrak C}(T')$ that coincides with $\tilde A^{\mathfrak C}(f)$.

  Let us check the equality $\tilde P^{\mathfrak C}_{T'}|G=\tilde A^{\mathfrak C}(f)\circ\tilde P^{\mathfrak C}_{T}$. It is enough to verify that the terms in the left and the right part of the equality which correspond to some tribe $C\in\mathfrak C(T)_{prim}^*$, are equal.

  Let $C'\in\mathfrak C(T')_{prime}^*$ and $C=C'\cap G$. If $C'\not\supset G'\setminus G$ then $C'=C$, hence, $\lfloor\frac{|C|}2\rfloor=\lfloor\frac{|C'|}2\rfloor$ and $b(g,\chi_C)=b'(g,\chi_{C'})$ for any $g\in G$. Thus, the correspondent terms of the tribes $C$ and $C'$ in $\tilde P^{\mathfrak C}_{T}(g)$ and $\tilde P^{\mathfrak C}_{T'}(g)$ are equal.

  Let $C'\supset G'\setminus G$. Since $C'\ne C_0(T')$, the operation $f$ is of type $M_3$. Then $|C'|=|C|+2$ and for any $g\in G$ $b'(g,\chi_{C'})=b(g,\chi_C)+ b(g,s)$. Hence, for any $g\in G$

  \begin{multline*}
  b'(g,\chi_{C'})-\left\lfloor\frac{|C'|}2\right\rfloor\cdot b'(g,s)=\\
  b(g,\chi_C)+ b(g,s)-\left\lfloor\frac{|C|}2\right\rfloor\cdot b(g,s)-b(g,s)=\\
  b(g,\chi_C)-\left\lfloor\frac{|C|}2\right\rfloor\cdot b(g,s).
  \end{multline*}
  Thus, the correspondent terms of the tribes $C$ and $C'$ in $\tilde P^{\mathfrak C}_{T}(g)$ and $\tilde P^{\mathfrak C}_{T'}(g)$ are equal.
\end{proof}

\begin{definition}\label{def:reduced_stable_parity_functor}
The reduces parity functor $\tilde P^{st}=\tilde P^{\mathfrak C_\infty}$ associated with the stable tribal system $\mathfrak C_\infty$ is called the {\em reduced stable parity functor}.
\end{definition}

\begin{example}\label{exa:reduced_stable_parity_functor}
Consider the based matrix $T$ with coefficients in $\Z_2$ from Example~\ref{exa:stable_parity_functor}. One of primitive based matrices homologus to $T$ is $T_\bullet=(\{s,7,8\},s,b_\bullet)$ where $b_\bullet$ is defined by the matrix $B_\bullet$:

$$
B_\bullet=
\begin{blockarray}{ccc}
s & \mathbf{7} & \mathbf{8}\\
\begin{block}{(c|c|c)}
  0 & 0 & 0\\
  \BAhline
  0 & 0 & 1\\
  \BAhline
  0 & 1 & 0\\
\end{block}
\end{blockarray}.
$$

Thus, the primitive tribes in $\mathfrak C_\infty(T)$ are $\{s\}$, $C_1=\{5,6,7\}$ and $C_2=\{8\}$. The zero tribe $C_0=\{3,4\}$ is not primitive. The element $\tilde b(s)$ is $0$. Then the coefficient group $\tilde A^{st}(T)$ of the reduced stable  parity functor is isomorphic to $\Z_2\oplus\Z_2\oplus\Z_2$ (the first summand is for $\{s\}$). The parity values correspond to the elements in the first and the last two colums of the matrix $B_3$. We extracted those columns to the matrix $\tilde B_3$.

$$
\tilde B_3=
\left(
\begin{array}{c|c|c}
 0 &    0  & 0\\
 \hline
 0 &   0  & 0\\
 0 &     0  & 0\\
 \hline
 0 &    0  & 0\\
 0 &    0  & 0\\
 \hline
 0 &    0  & 1\\
 0 &     0  & 1\\
 0 &     0  & 1\\
 \hline
 0 &    1  & 0
\end{array}
\right).
$$

Thus, the parities of the elements are
\begin{gather*}
\tilde P^{st}(\mathbf{1})=\tilde P^{st}(\mathbf{2})=\tilde P^{st}(\mathbf{3})=\tilde P^{st}(\mathbf{4})=(0,0,0),\\
\tilde P^{st}(\mathbf{5})=\tilde P^{st}(\mathbf{6})=\tilde P^{st}(\mathbf{7})=(0,0,1),\quad \tilde P^{st}(\mathbf{8})=(0,1,0).
\end{gather*}
\end{example}

\begin{remark}
Although the coefficient groups $\tilde A^{\mathfrak C}(T)$ of the reduced  parity functor $\tilde P^{\mathfrak C}$ are all isomorphic we can not state that the parity functor defines a parity with coefficients in $\tilde A^{\mathfrak C}(T)$ for some fixed $T$.

Consider the following example. Let $T=({s,x,y},s,b)$ be a based matrix with coefficients in $\Z_2$ where $b$ is determined be the matrix

$$
B(T)=
\begin{blockarray}{ccc}
s & x & y\\
\begin{block}{(c|cc)}
  0 & 0 & 0\\
  \BAhline
  0 & 0 & 1\\
  0 & 1 & 0\\
\end{block}
\end{blockarray}.
$$

The based matrix $T$ is primitive and has two primitive stable tribes $\{x\},\{y\}\in \mathfrak C_\infty(T)_{prim}$.

Replace the element $x$ with an element $u$ by applying operations $M_3$ and $M_3^{-1}$ as shown below. Thus, we get a based matrix $T'=(\{s,u,y\},s,b')$.

$$
\begin{blockarray}{ccc}
s & x & y\\
\begin{block}{(c|cc)}
  0 & 0 & 0\\
  \BAhline
  0 & 0 & 1\\
  0 & 1 & 0\\
\end{block}
\end{blockarray}
\quad \Longrightarrow \quad
\begin{blockarray}{ccccc}
s & x & p & u & y\\
\begin{block}{(c|cccc)}
  0 & 0 & 0 & 0 & 0\\
  \BAhline
  0 & 0 & 0 & 0 & 1\\
  0 & 0 & 0 & 0 & 1\\
  0 & 0 & 0 & 0 & 1\\
  0 & 1 & 1 & 1 & 0\\
\end{block}
\end{blockarray}
\quad \Longrightarrow \quad
\begin{blockarray}{ccc}
s & u & y\\
\begin{block}{(c|cc)}
  0 & 0 & 0\\
  \BAhline
  0 & 0 & 1\\
  0 & 1 & 0\\
\end{block}
\end{blockarray}
$$

Analogously, replace $y$ with $v$. Then we get a sequence of transformations $f_1\colon T\to T_1=(\{s,u,v\},s,b_1)$. The primitive stable tribes of $T_1$ are $\{u\}$ and $\{v\}$. The transformation $f_1$ identifies $x$ with $u$ and $y$ with $v$.

On the other hand, we can replace $x$ with $v$ and $y$ with $u$ and get a based matrix $T_2=(\{s,u,v\},s,b_2)$. The transformation $f_2$ from $T$ to $T_2$ identifies the primitive tribes $x$ with $v$ and $y$ with $u$.

But the based matrices $T_1$ and $T_2$ coincide because they have the same underlying set $\{s,u,v\}$ and $b_1=b_2$. Thus, there is no canonical way to identify primitive tribes of homologous based matrices. The reason is we have a nontrivial automorphism of the based matrix $T$. In general case we can assign a {\em monodromy}, i.e. a sequence of transformations from a based matrix to itself, which interchanges the primitive tribes according to any given automorphism of the based matrix. On the other hand, any monodromy of a primitive based matrix induces an automorphism of the based matrix~\cite{T}.

Thus, we need to factorise by isomorphisms if we want to get an invariant coefficient group for a parity.
\end{remark}

\subsection{Reduced parity}

Let $T_0=(G_0,s_0,b_0)$ be a based matrix with coefficients in $H$ and let $\mathfrak C$ be a tribal system defined on the based matrices $T=(G,s,b)$ homologous to $T_0$.

Choose a primitive based matrix $T_\bullet=(G_\bullet,s,b_\bullet)$ homologous to $T_0$. The tribes in $G_\bullet$ correspond to the primitive tribes of any based matrix $T=(G,s,b)$ homologous to $T_\bullet$. Let $Aut(T_\bullet)$ be the automorphism group of the based matrix $T_\bullet$. Note that the group $Aut(T_\bullet)$ considered up to isomorphism does not depend on the choice of a primitive based matrix in the homology class.

Let us define $\bar{\mathfrak C}(T_\bullet)$ to be the finest partition such that $\mathfrak C(T_\bullet)\succ \bar{\mathfrak C}(T_\bullet)$ and $\phi({\mathfrak C}(T_\bullet))=\bar{\mathfrak C}(T_\bullet)$ for any $\phi\in Aut(T_\bullet)$. Note that $\{s\}\in \bar{\mathfrak C}(T_\bullet)$. The partition $\bar{\mathfrak C}(T_\bullet)$ defines a partition on the set ${\mathfrak C}(T_\bullet)_{prim}={\mathfrak C}(T_\bullet)$. By Proposition~\ref{prop:primitive_tribe_invariance} this partition can be lifted to some partition of the set $\bar{\mathfrak C}(T)_{prim}$ for any based matrix $T$ homologous to $T_\bullet$, and the partition does not depend on the sequence of transformations connecting $T$ and $T_\bullet$.

For a based matrix $T=(G,s,b)$ denote $\bar{\mathfrak C}(T)_{prim}^*=\bar{\mathfrak C}(T)_{prim}\setminus\{\bar C_0(T)\}$ where $\bar C_0(T)$ is the element of the partition $\bar{\mathfrak C}(T)_{prim}$ which includes the zero tribe $C_0(T)$ if it is primitive, and $\bar C_0(T)=\emptyset$ if $C_0(T)$ is not primitive. Let
$\bar A^{\mathfrak C}(T)=H[\bar{\mathfrak C}(T)_{prim}^*]/\langle\bar b(s)\rangle$
where
\begin{equation}\label{eq:bar_bs_element}
\bar b(s)=\sum_{\bar C\in\bar{\mathfrak C}(T)_{prim}^*} b(s,\chi_{\bar C})\cdot \bar C.
\end{equation}
Define the maps
$\bar P^{\mathfrak C}_T\colon G\to \bar A^{\mathfrak C}(T)$ by the formula
\begin{equation}\label{eq:reduced_parity}
\bar P^{\mathfrak C}_T(g)= \sum_{\bar C\in\bar{\mathfrak C}(T)_{prim}^*} \left(b(g,\chi_{\bar C})-\left\lfloor\frac{|\bar C|}2\right\rfloor\cdot b(g,s)\right)\cdot \bar C.
\end{equation}

For an operation $f\colon T\to T'$, $T=(G,s,b)$, $T'=(G',s,b')$, of type $M_1, M_2$ or $M_3$  we define a homomorphism $\bar A^{\mathfrak C}(f)\colon \bar A^{\mathfrak C}(T)\to \bar A^{\mathfrak C}(T')$ by the formula
\begin{equation}\label{eq:reduced_parity_coefficients}
\bar A^{\mathfrak C}(f)\left(\sum_{\bar C\in\bar{\mathfrak C}(T)_{prim}^*}\lambda_{\bar C}\cdot \bar C\right)=\sum_{\bar C\in\bar{\mathfrak C}(T)_{prim}^*}\lambda_{\bar C}\cdot f(\bar C)
\end{equation}
where $f(\bar C)\in \bar{\mathfrak C}(T')_{prim}^*$ is the unique class such that $f(\bar C)\cap G = \bar C$.

\begin{theorem}\label{thm:reduced_parity_tribal_system}

\begin{enumerate}
  \item Let $T=(G,s,b)$ be a based matrix homologous to $T_\bullet$ and $f\colon T\to T_\bullet$ be a transformation chain connecting these based matrices. Then the map $\bar A^{\mathfrak C}(f)\colon \bar A^{\mathfrak C}(T)\to \bar A^{\mathfrak C}(T_\bullet)$ is an isomorphism which does not depend on the choice of the transformation $f$.
  \item The family of maps $\bar p^{\mathfrak C}_T=\bar A^{\mathfrak C}(f)\circ \bar P^{\mathfrak C}_T\colon G\to \bar A^{\mathfrak C}(T_\bullet)$ is a parity with coefficients in $\bar A^{\mathfrak C}(T_\bullet)$.
\end{enumerate}
\end{theorem}

The parity $\bar p^{\mathfrak C}$ is called the {\em reduced parity associated with the tribal system $\mathfrak C$}.

\begin{proof}
1. We can prove that $\bar A^{\mathfrak C}(f)$ using the arguments of Theorem~\ref{thm:reduced_parity_functor_based_matrices}. Let $f_i\colon T\to T_\bullet$ be two sequences connecting $T$ and $T_\bullet$. Then $\phi=f_2\circ f_1^{-1}$ is a sequence of transformations of $T_\bullet$ to itself. Since $T_\bullet$ is primitive, $\phi$ is an automorphism of $T_\bullet$. Hence, $\phi$ acts trivially on $\bar{\mathfrak C}(T_\bullet)_{prim}^*$. Therefore $\phi$ induces the identity automorphism of $\bar A^{\mathfrak C}(T_\bullet)$, i.e. $\bar A^{\mathfrak C}(f_2\circ f_1^{-1})=id$. Hence, $\bar A^{\mathfrak C}(f)=\bar A^{\mathfrak C}(f_2\circ f_1^{-1})\cdot \bar A^{\mathfrak C}(f_1)=\bar A^{\mathfrak C}(f_1)$.

2. Following the arguments of Theorem~\ref{thm:reduced_parity_functor_based_matrices} we can show that $\bar P^{\mathfrak C}$ is a parity functor. We need only to check the equality $\bar p^{\mathfrak C}_{T'}|_G=\bar p^{\mathfrak C}_T$ for any based matrices $T=(G,s,b)$ and $T'=(G',s,b')$ connected by an operation $f\colon T\to T'$ of type $M_i$, $i=1,2,3$.

Let $f'\colon T'\to T_\bullet$ be a sequence of transformations $M_i^{\pm 1}$, $i=1,2,3$ from $T'$ to $T_\bullet$. By definition
\begin{multline*}
\bar p^{\mathfrak C}_{T'}|_G=\bar A^{\mathfrak C}(f')\circ\bar P^{\mathfrak C}_{T'}|_G=\bar A^{\mathfrak C}(f')\circ\bar A^{\mathfrak C}(f)\circ\bar P^{\mathfrak C}_{T}=\bar A^{\mathfrak C}(f'\circ f)\circ\bar P^{\mathfrak C}_{T}=\bar p^{\mathfrak C}_{T}
\end{multline*}
where the second equality is the parity functor property.
\end{proof}

\begin{definition}
The reduced parity $\hat p^{st}=\hat p^{\mathfrak C_\infty}$ associated with the stable tribal system $\mathfrak C_\infty$ is called the {\em reduced stable parity}.
\end{definition}

\begin{example}
Consider the based matrix $T=(\{s,1,2,3,4,5,6,7,8\},s,b)$ from Example~\ref{exa:stable_parity_functor}. A primitive based matrix $T_\bullet=(\{s,7,8\},s,b_\bullet)$ is given in Example~\ref{exa:reduced_stable_parity_functor}. It has three primitive stable tribes $\{s\}$, $\{7\}$ and $\{8\}$. The automorphism group of $T_\bullet$ is $Aut(T_\bullet)=\Z_2$, it interchanges the primitive tribes $7$ and $8$. Then the partition $\bar{\mathfrak C}_\infty(T_\bullet)_{prim}$ consists of the element $\{7,8\}$ and $\{s\}$. The coefficient group $\bar A^{st}$ is equal to $\Z_2\oplus\Z_2$ since the zero tribe is not primitive and $\bar b(s)=0$.

The induced partition for $T$ will be $\bar{\mathfrak C}_\infty(T)_{prim}=\{\{s\},\{5,6,7,8\}\}$. It corresponds to the matrix $\bar B_3$ which is obtained from $\tilde B_3$ by taking the sum of the last two columns.

$$
\bar B_3=
\left(
\begin{array}{c|c}
 0   & 0\\
 \hline
 0   & 0\\
 0   & 0\\
 \hline
 0   & 0\\
 0   & 0\\
 \hline
 0   & 1\\
 0   & 1\\
 0  & 1\\
 \hline
 0 &    1
\end{array}
\right).
$$

Then the reduced stable parity on $T$ is equal to
\begin{gather*}
\bar p^{st}(\mathbf{1})=\bar p^{st}(\mathbf{2})=\bar p^{st}(\mathbf{3})=\bar p^{st}(\mathbf{4})=(0,0),\\
\bar p^{st}(\mathbf{5})=\bar p^{st}(\mathbf{6})=\bar p^{st}(\mathbf{7})=\bar p^{st}(\mathbf{8})=(0,1).
\end{gather*}
\end{example}

\subsection{Reduced stable parity for virtual and flat knot}

Let $\mathcal K$ be a virtual or flat knot. Then we can define {\em reduced stable parity} on the diagrams of the knot $\mathcal K$ as follows. Let $D$ be a diagram of $\mathcal K$. Consider the based matrix $T(D)$ of the diagram $D$ and a primitive based matrix $T_\bullet(D)$ homologous to $T(D)$.
%Find the stable tribes on $T_\bullet(D)$ and construct the partition $\bar{\mathfrak C}_\infty(T_\bullet(D))$ invariant under the automorphism of
Denote the group $\bar A^{st}(T_\bullet(D))$ by $\bar A^{st}(\mathcal K)$.
%is the coefficient group of the reduced stable parity, and the map $\hat p^{st}_T(D)$ is the parity map for the crossing of the diagram $D$.

\begin{proposition}\label{prop:reduced_stable_parity_virtual_knots}
The family of  the maps $\bar p^{st}_{T(D)}\colon \V(D)\to\bar A^{st}(\mathcal K)$ is a parity on diagrams of the knot $\mathcal K$ with coefficients in the group $\bar A^{st}(\mathcal K)$.
\end{proposition}

Proposition~\ref{prop:reduced_stable_parity_virtual_knots} follows from Lemma~\ref{lem:funtoriality_based_matrix} and Theorem~\ref{thm:reduced_parity_tribal_system}.

\begin{definition}
The parity $\bar p^{st}$ from Proposition~\ref{prop:reduced_stable_parity_virtual_knots} is called the {\em reduced stable parity} on virtual (flat) knots.
\end{definition}

\begin{example}
Consider the flat knot diagram $D$ in Fig.~\ref{fig:flat_knot_diagram1}. Then $T(D)=(\{s,1,2,3\}, s, b)$ where $b$ is defined by the matrix $B$ with coefficients in $\Z_2$.

\begin{figure}[h]
\centering
\raisebox{-0.5\height}{\includegraphics[width=0.15\textwidth]{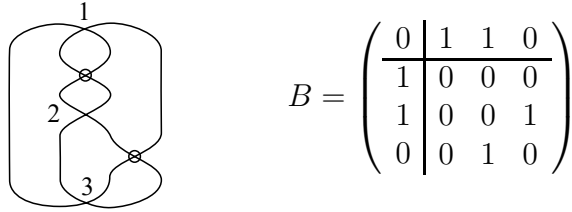}}\qquad\qquad
$
B= \left(\begin{array}{c|ccc}
0 & 1 & 1 & 0 \\
\hline
1 & 0 & 0 & 0 \\
1 & 0 & 0 & 1 \\
0 & 0 & 1 & 0
\end{array}\right)
$

\caption{Flat knot diagram and the corresponding based matrix}\label{fig:flat_knot_diagram1}
\end{figure}

The based matrix $T(D)$ is primitive. Its stable tribes are
$${\mathfrak C}_\infty = \{\{s\},\{1\},\{2\},\{3\}\}.
$$
Since the automorphism group $Aut(T(D)$ is trivial and the zero tribe is empty, the partition $\bar{\mathfrak C}_\infty(T(D))_{prim}$ coincides with ${\mathfrak C}_\infty$. Thus, we have four reduced stable classes. The element $\bar b(s)=(0,1,1,0)\in\Z_2[{\mathfrak C}_\infty]$. Hence,
 $\bar A^{st}=\Z_2[{\mathfrak C}_\infty]/\langle\bar b(s)\rangle\simeq\Z_2^3$ where the isomorphism is induced by the map $$\lambda_s\cdot\mathbf{s}+\lambda_1\cdot\mathbf{1}+\lambda_2\cdot\mathbf{2}+\lambda_3\cdot\mathbf{3}\mapsto (\lambda_s,\lambda_1+\lambda_2,\lambda_3).$$

 The parity values are
 \begin{gather*}
   \bar p^{st}(1)=\left( 1, \left(0-\left\lfloor\frac 12\right\rfloor 1\right)+\left(0-\left\lfloor\frac 12\right\rfloor 1\right),0-\left\lfloor\frac 12\right\rfloor 1\right)=(1,0,0),\\
   \bar p^{st}(2)=\left( 1, \left(0-\left\lfloor\frac 12\right\rfloor 1\right)+\left(0-\left\lfloor\frac 12\right\rfloor 1\right),1-\left\lfloor\frac 12\right\rfloor 1\right)=(1,0,1),\\
   \bar p^{st}(3)=\left( 0, \left(0-\left\lfloor\frac 12\right\rfloor 0\right)+\left(1-\left\lfloor\frac 12\right\rfloor 0\right),0-\left\lfloor\frac 12\right\rfloor 0\right)=(0,1,0).
\end{gather*}

\end{example}

\appendix

\section{Reduced stable parity of knots of complexity $\le 4$}

Using a Mathematica program, we calculated the reduced stable parity for the diagrams of Green's table~\cite{Gr} with the number of crossings less or equal $4$.

It turns out that the based matrices of all these diagrams are either primitive or reduce to the trivial matrix $(0)$.

In the last case there are no primitive tribes and the reduced stable parity $\bar p^{st}$ coincides with the index parity $ip$. This situation occurs for the knots
\begin{gather*}
2.1, 3.2, 3.5, 3.6, 3.7, 4.3, 4.6, 4.12, 4.25, 4.27, 4.36, 4.37, 4.40,\\
 4.41, 4.43, 4.44, 4.46, 4.53, 4.54, 4.61, 4.64, 4.65, 4.68, 4.73,\\
 4.74, 4.75, 4.82, 4.84, 4.86, 4.91, 4.92, 4.94, 4.95, 4.96, 4.99, \\
 4.100, 4.101, 4.102, 4.104, 4.105, 4.106, 4.108.
\end{gather*}

The knots $3.1, 3.3, 3.4$ have three primitive non zero tribes corresponding to the crossings and the coefficient group $\bar A^{st}$ is isomorphic to $\Z^3$.

For the knots
\begin{gather*}
4.2, 4.5, 4.7, 4.10, 4.11, 4.14, 4.15, 4.17, 4.19, 4.20, 4.21, 4.22, 4.23, 4.24,\\
4.26, 4.28, 4.29, 4.30, 4.32, 4.34, 4.35, 4.38, 4.39, 4.42, 4.45, 4.47, 4.48,\\
4.49, 4.50, 4.56, 4.57, 4.59, 4.62, 4.63, 4.66, 4.67, 4.70, 4.71, 4.76, 4.78, \\
4.79, 4.80, 4.81, 4.83, 4.87, 4.88, 4.93, 4.97, 4.103,
\end{gather*}
each of 4 crossings determines a primitive non zero tribe, and the coefficient group $\bar A^{st}$ of the reduced stable parity is isomorphic to $\Z^4$.

The knots $4.9, 4.16, 4.33, 4.52, 4.58, 4.72$ have 3 primitive tribes, and $\bar A^{st}=\Z^4$ for them.

For example, the based matrix of the knot $4.9$ is equal to
\[
\left(
\begin{array}{c|cccc}
 0 & 1 & 0 & 0 & -1 \\
 \hline
 -1 & 0 & -1 & -1 & 0 \\
 0 & 1 & 0 & -1 & -1 \\
 0 & 1 & 1 & 0 & -1 \\
 1 & 0 & 1 & 1 & 0 \\
\end{array}
\right).
\]
The crossings $1$ and $4$ belongs to one tribe because the corresponding rows (or columns) are proportional. Note that $1$ and $4$ are not complementary. The parity matrix is given below.
\[
\left(
\begin{array}{c|ccc}
 0 & 0 & 0 & 0 \\
 \hline
 -1 & 1 & -1 & -1 \\
 0 & 0 & 0 & -1 \\
 0 & 0 & 1 & 0 \\
 1 & -1 & 1 & 1 \\
\end{array}
\right)
\]
The first column corresponds to $s$ (and the index parity), the other columns correspond to non zero primitive tribes. The first row presents the vector $\bar b(s)$, the other rows present the parity values of the crossings.

The knots $4.85, 4.89, 4.90, 4.98, 4.107$ have 4 primitive tribes, and $\bar A^{st}=\Z^4\oplus\Z_2$ for them.

For example, the knot $4.85$ has the based matrix
\[
\left(
\begin{array}{c|cccc}
 0 & 2 & -2 & -2 & 2 \\
 \hline
 -2 & 0 & -2 & -1 & 0 \\
 2 & 2 & 0 & 0 & 3 \\
 2 & 1 & 0 & 0 & 2 \\
 -2 & 0 & -3 & -2 & 0 \\
\end{array}
\right)
\]
which coincides with the parity matrix. Each vertex defines a primitive tribe. The vector $\bar b(s)=(0,2,-2,-2,2)$ is divisible by $2$. Hence, the coefficient group is equal to
\[
\bar A^{st}=\Z[\bar{\mathfrak C}_\infty(T(D))_{prim}]/\left<\bar b(s)\right>=\Z^5/\left<(0,2,-2,-2,2)\right>=\Z^4\oplus\Z_2.
\]

The knots $4.13, 4.18, 4.31, 4.51, 4.60, 4.69$ have 3 primitive tribes, the symmetry group $Aut(T_\bullet)=\Z_2$ which acts trivially on the primitive tribes, and $\bar A^{st}=\Z^3$.

For example, the knot $4.13$ has the based matrix
\[
\left(
\begin{array}{c|cccc}
 0 & -1 & 0 & 0 & 1 \\
 \hline
 1 & 0 & 1 & 1 & 1 \\
 0 & -1 & 0 & 0 & 0 \\
 0 & -1 & 0 & 0 & 0 \\
 -1 & -1 & 0 & 0 & 0 \\
\end{array}
\right)
\]
The primitive tribes are $\{1\}, \{2,3\}$ and $\{4\}$ (the rows the second and third crossings coincide). The symmetry group acts by transposition of the second and the third crossings, and induces trivial action on the primitive tribes. The parity matrix is
\[
\left(
\begin{array}{c|ccc}
 0 & -1 & 0 & 1 \\
 \hline
 1 & 0 & 1 & 1 \\
 0 & -1 & 0 & 0 \\
 0 & -1 & 0 & 0 \\
 -1 & -1 & 1 & 0 \\
\end{array}
\right).
\]
Since the vector $\bar b(s)\ne 0$, the coefficient group is $\bar A^{st}=\Z^4/\Z=\Z^3$.

Finally, the knots $4.1, 4.4, 4.8, 4.55, 4.77$ have 4 primitive tribes and the symmetry group $Aut(T_\bullet)=\Z_2$ which interchange pairs of crossings. The coefficient group is equal to $\bar A^{st}=\Z^2\oplus\Z_2$.

For example, the knot $4.1$ has the based matrix
\[
\left(
\begin{array}{c|cccc}
 0 & 1 & -1 & 1 & -1 \\
 \hline
 -1 & 0 & -1 & 0 & 0 \\
 1 & 1 & 0 & 0 & 0 \\
 -1 & 0 & 0 & 0 & -1 \\
 1 & 0 & 0 & 1 & 0 \\
\end{array}
\right).
\]
Each crossing defines a primitive tribe. There is a unique nontrivial symmetry $(13)(24)$ of the based matrix. Hence, $\bar{\mathfrak C}(T)_{prim}=\{\{1,3\},\{2,4\}\}$. The parity matrix looks like
\[
\left(
\begin{array}{c|cc}
 0 & 2 & -2 \\
 \hline
 -1 & 1 & 0  \\
 1 & 0 & -1  \\
 -1 & 1 & 0  \\
 1 & 0 & -1  \\
\end{array}
\right).
\]
The coefficient groups is equal $\bar A^{st}=\Z^3/\left<(0,2,-2)\right>=\Z^2\oplus\Z_2$.

%\end{frame}


\begin{thebibliography}{99}

\bibitem{C} P. Cahn, A generalization of Turaev's virtual string obracket and self-intersections of virtual strings, {\em Communications in Contemporary Mathematics} {\bf 19}:4 (2017) 1650053.

\bibitem{F} D. Freund, Multistring based matrices {\em J. Knot Theory Ramifications} {\bf 29}:6 (2020) 2050038.

%\bibitem{IM2} D.~Ilyutko, V.O.~Manturov, Picture-valued parity-biquandle bracket, {\em J. Knot Theory Ramifications} {\bf 29}:2 (2020) 2040004.

%\bibitem{IM} D.\,P.~Ilyutko, V.\,O.~Manturov, Cobordisms of Free Knots, {\em Doklady Mathematics} {\bf 80}:3 (2009) 1--3.

%\bibitem{IM2} D.P.~Ilyutko, V.O.~Manturov, A parity map of framed chord diagrams, {\em J. Knot Theory Ramifications} {\bf 24}:13 (2015) 1541006.

\bibitem{Gr}{J. Green, Table of virtual knots,\\ https://www.math.toronto.edu/drorbn/Students/GreenJ/}


\bibitem{H} A. Henrich, A sequence of degree one vassiliev invariants for virtual knots, {\em J. Knot Theory Ramifications} {\bf 19}:4 (2010) 461--487.

\bibitem{IMN} D.~Ilyutko, V.O.~Manturov, I.~Nikonov, {\em Parity and Patterns in Low-Dimensional Topology}, Cambridge Scientific Publishers, 2015.

\bibitem{KK}N. Kamada, S. Kamada, Abstract link diagrams and virtual knots, {\em J. Knot Theory Ramifications} {\bf 9}:1 (2000) 93--106.

\bibitem{K} L.H.~Kauffman, Virtual knot theory, {\em Europ. J. Combinatorics} {\bf 20}:7 (1999) 663--690.

\bibitem{M1} V.O.~Manturov, Parity in knot theory, {\em Sb. Math.} {\bf 201}:5 (2009) 693--733.

%\bibitem{M4} V.~O.~Manturov, Parity, free knots, groups, and invariants of finite type, {\em Trans. Moscow Math. Soc.} (2011) 157--169.

%\bibitem{M5} V.~O.~Manturov, Parity and cobordisms of free knots, {\em Math. sb.} {\bf 203}:5 (2012) 196--223.

% \bibitem{M6} V.~O.~Manturov, Free knots and parity, in {\em Introductory Lectures on Knot Theory, Selected Lectures Presented at the Advanced
%School and Conference on Knot Theory and its Applications to Physics and Biology, Series of Knots and Everything}, {\bf 46} (2012) 321--345.

%\bibitem{M7} V.\,O.~Manturov, Parity and projection from virtual knots to classical knots, {\em J. Knot Theory Ramifications} {\bf
%22}:9 (2013) 1350044.



\bibitem{N} I.~Nikonov, Parity functors, arxiv:2109.12230.

\bibitem{P} N. Petit, Finite-type invariants of long and framed virtual knots, {\em J. Knot Theory Ramifications} {\bf 28}:10 (2019) 1950054.

\bibitem{T} V.~Turaev, Virtual strings, {\em Ann. Inst. Fourier} {\bf 54}:7  (2004) 2455--2525.

\bibitem{T2} V.~Turaev, Cobordism of knots on surfaces, {\em J. Topol} {\bf 1}:2 (2008) 285--305.





\end{thebibliography}
\end{document}